\documentclass[10pt]{article} 
\usepackage[margin=1in]{geometry}

 \usepackage{comment}
\usepackage{listings}
\usepackage{geometry}
\usepackage{cite}
\usepackage{subcaption}

\usepackage{amsmath,amsthm,amsfonts,amssymb,amscd}
\usepackage[colorlinks=true, pdfstartview=FitV, linkcolor=blue,citecolor=blue, urlcolor=blue]{hyperref}
\usepackage[usenames,dvipsnames]{color}
\usepackage{times}

\usepackage{mathtools}
\usepackage{mathabx}
\usepackage{float}
\usepackage{makeidx}
\usepackage{stmaryrd}
\usepackage{mathrsfs}
\usepackage{emptypage}
\usepackage{xfrac}
\usepackage{bbm}
\def\dd{\,\mathrm{d}}

\renewcommand{\epsilon}{\varepsilon}

\usepackage{graphicx}
\newcommand{\p}{\ensuremath{\partial}}

\definecolor{labelkey}{rgb}{0,0,1}

\newcommand{\R}{\mathbb{R}}

\newcommand{\Z}{\mathbb{Z}}

\newcommand{\al}{\alpha}

\newcommand{\RR}{\mathbb R}

\newcommand{\TT}{\mathbb T}
\newcommand{\DD}{\mathbb D}

\renewcommand*{\tilde}{\widetilde}

\renewcommand*{\bar}{\overline}

\newcommand{\Omab}{\Omega_a^b}
\newcommand{\ta}{\widetilde{a}}
\newcommand{\pd}{\partial}
\newcommand\D{_{\mathrm{D}}}
\newcommand\DN{_{\mathrm{DN}}}
\newcommand{\cX}{{\mathcal X}}
\newcommand{\cY}{{\mathcal Y}}

\newcommand{\T}{{\mathbb T}}

\newtheorem{theorem}{Theorem}[section]
\newtheorem{lemma}[theorem]{Lemma}

\newtheorem{proposition}[theorem]{Proposition}
\newtheorem{corollary}[theorem]{Corollary}
\theoremstyle{definition}
\newtheorem{definition}[theorem]{Definition}

\newtheorem{remark}[theorem]{Remark}

\numberwithin{equation}{section}

\def\p{\partial}

\def\f1r{{\frac{1}{r}}  }
\def\p{\partial}

\def\f1r{{\frac{1}{r}}  }

\allowdisplaybreaks[4]

\title{\bf A contractible Schiffer counterexample on the half-sphere}
 \author{Gonzalo Cao-Labora and Antonio J. Fern\'andez}

\date{} %
\begin{document}

\maketitle
\begin{abstract} 
  We show the existence of a family of nontrivial smooth contractible domains $\Omega \subset \mathbb S^2$ that admit Neumann eigenfunctions of the Laplacian which are constant on the boundary. These domains are contained on the half-sphere $\mathbb S^2_+$, in stark contrast with the rigidity literature for Serrin-type problems. The proof relies on a local bifurcation argument around the family of geodesic disks centered at the north pole. We combine the use of anisotropic H\"older spaces for the functional setting with computer-assisted techniques to check the bifurcation conditions. 
\end{abstract}

\section{Introduction and main results}

In the last few years, there has been renewed interest in the so-called Schiffer conjecture. This conjecture, which can be traced back to the 1950s, was stated in the 1982 famous list of open problems by S.-T. Yau as follows:

\begin{itemize}  
\item[]\textbf{\hspace{-0.03cm}\cite[Problem 60]{Yau} } \textit{Let $\Omega$ be a smooth, compact bounded domain in $\mathbb{R}^2$. Suppose there exists an $f$ which is an eigenfunction for the Laplacian with Neumann boundary conditions. If also $f$ is constant on the boundary of $\Omega$, prove $\Omega$ is a disk.}
\end{itemize}

Although, after 75 years, the validity of the Schiffer conjecture remains unknown, in the case where $\Omega \subset \R^2$ is simply connected, some partial rigidity results are available. For example, it is known that the existence of an infinite sequence of orthogonal Neumann eigenfunctions that are constant on $\partial \Omega$ implies that $\Omega$ must be a disk (see \cite{B,BY}). Also, if the eigenvalue is among the seven lowest Neumann eigenvalues of the domain, then $\Omega$ must as well be a disk (see \cite{A, D}). See the references in \cite{FMW, EFRS} for more results in this direction. We stress that contractible counterexamples to the Schiffer conjecture are of paramount significance, since Williams \cite{Williams} showed that in that case, the Schiffer conjecture is equivalent to the Pompeiu problem \cite{Pompeiu}. The Pompeiu problem is a $100$-year-old open problem asking if the integrals of a continuous function $f$ over every rigid motion of $\Omega$ uniquely determine $f$.

Despite the aforementioned rigidity results for the Schiffer conjecture, in the recent papers \cite{FMW, EFRS}, the authors dealt with slightly relaxed versions of this conjecture and proved suitable counterexamples. 

First, Fall, Minlend and Weth \cite{FMW} proved that the conjecture cannot hold if one removes the boundedness assumption from $\Omega$. More precisely, they proved the existence of a parametric family of compact subdomains $\Omega$ of the strip $\R \times \T$ which admit Neumann eigenfunctions with constant Dirichlet values on $\pd \Omega$. Moreover, the boundaries of these domains have nonconstant principal curvatures. They actually dealt with the more general case of compact subdomains of the flat cylinder $\R^n \times \TT$, with $n \geq 1$. Note that here and in the rest of the paper we use the notation $\T := \R/2\pi\Z$. 

On the other hand, if one considers non-simply connected domains, it is natural to wonder what happens if one relaxes the hypotheses in the conjecture by allowing the eigenfunction to be locally constant on the boundary, that is, constant on each connected component of~$\partial\Omega$. Of course, the radial Neumann eigenfunctions of a ball or an annulus are locally constant on the boundary, so the natural question in this case is whether $\Omega$ must necessarily be a ball or an annulus. This question shares many features with the Schiffer conjecture; essentially the same rigidity results available for Schiffer hold in this setting \cite{EFRS}. However, Enciso, the second author, Ruiz and Sicbaldi \cite{EFRS} constructed parametric families of doubly connected bounded domains $\Omega \subset \R^2$ such that the Neumann eigenvalue problem in $\Omega$ admits a non-radial eigenfunction that is locally constant on the boundary $\pd \Omega$.

More than 25 years ago, Shklover \cite{S} considered a generalized version of the Schiffer problem. Assuming that $(M,g)$ is a complete, analytic Riemannian manifold of dimension $n \geq 2$, he considered the overdetermined problem
\begin{equation} \label{E.OverdeterminedM}
\left\{
\begin{aligned}
\ & \Delta_g u +\mu u = 0 && \textup{in } \Omega\,,\\
& u = {\rm constant} \,, \quad \partial_{\nu} u  = 0 \quad&& \textup{on } \partial \Omega\,.
\end{aligned}
\right.
\end{equation}
Here, $\Omega \subset M$ is a smooth bounded domain with $\pd \Omega$ connected, $\mu > 0$, and $\Delta_g$ denotes the Laplace-Beltrami operator on the manifold. He then claimed \cite[Page 540]{S} that the natural generalization of the Schiffer conjecture to this setting should be as follows: \emph{the existence of a nontrivial solution to \eqref{E.OverdeterminedM} implies that $\pd \Omega$ is homogeneous\footnotemark{} in $M$}.  Nevertheless, he provided a counterexample to this generalization. He constructed domains $\Omega \subset \mathbb{S}^n$, $n \geq 3$, where \eqref{E.OverdeterminedM} admits a solution $u \neq 0$, and whose boundary $\partial \Omega$ is not a homogeneous hypersurface in $\mathbb{S}^n$. In these domains, $\pd \Omega$ is however an isoparametric\footnotemark[\value{footnote}] hypersurface in $\mathbb{S}^n$. By Cartan's Theorem\footnotemark[\value{footnote}]\footnotetext{See Appendix \ref{A.geometry}.}, this implies that the principal curvatures of $\pd \Omega$ are constant. Hence, it is natural to wonder whether the principal curvatures of $\pd \Omega$ must be constant or not.

Later, Souam \cite{S2} studied \eqref{E.OverdeterminedM} in the case where $M = \mathbb{S}^2$ endowed with the standard round metric, but without assuming a priori that $\pd \Omega$ is connected. In this setting, he proved several rigidity results. On the one hand, he proved that if \eqref{E.OverdeterminedM} admits a non-trivial solution with $\mu = 2$, and $\Omega$ is simply connected, then $\Omega$ is a geodesic disk. On the other hand, he proved that the same conclusion holds if $\mu = \lambda_2(\Omega)$, the second Dirichlet eigenvalue of $\Delta_g$ on $\Omega$. Finally, he also formulated the following conjecture:

\begin{itemize}
\item[]\textbf{\hspace{-0.03cm}\cite[Conjecture 1.1]{S2} } \textit{Let $\Omega$ be a smooth domain on $\mathbb{S}^2$. Suppose that \eqref{E.OverdeterminedM} has a nontrivial solution. Then $\Omega$ is either a geodesic disk or a round symmetric annulus.}
\end{itemize}

This conjecture was very recently disproved by Fall, Minlend and Weth \cite{FMW}. They constructed a parametric family of domains which bifurcate from a round symmetric annulus in which \eqref{E.OverdeterminedM} admits nontrivial solutions. The two connected components of the boundaries of these domains have nonconstant principal curvatures. We refer to \cite[Theorem 1.5]{FMW} for a more precise statement. However, let us explicitly point out that these domains are not contractible and that they are not contained on the half-sphere $\mathbb{S}_+^2 := \{x \in \mathbb{S}^2: x_3 > 0\}$. 

Our first aim is to provide a counterexample to the Souam conjecture satisfying the aforementioned properties.  Before stating our main result, let us recall that a zonal function on $\mathbb{S}^2$ is a function that depends only on the colatitude.

\begin{theorem} \label{T.intro}
    There exists a parametric family of smooth contractible domains $\Omega \subset \mathbb{S}^2$, with $\overline{\Omega} \subset \mathbb{S}_+^2$, such that
    the Neumann eigenvalue problem
    $$
    \Delta u + \mu u = 0 \quad \textup{in } \Omega\,, \quad \pd_\nu u = 0 \quad \textup{on } \pd \Omega\,,
    $$
    admits, for some $\mu > 0$, a non-zonal eigenfunction that is constant on the boundary $\pd \Omega$. Here, $\Delta$ denotes the Laplace-Beltrami operator on $\mathbb{S}^2$ with respect to the standard round metric. 
    
    More precisely, for all $0 < |s| \ll 1$, the family of domains $\Omega \equiv \Omega^{b_s}$ is given in spherical coordinates by
$$
\Omega := \big\{(\theta,\phi) \in [0,\pi] \times \TT : \theta < \arccos(a_\star) + s b_s(\phi) \big\}\,,
$$
where $a_\star \in (\frac13, \frac12)$ and $b_s: \T \to \R$ is an analytic function of the form
$$
b_s(\phi) = b_\star \cos(8\phi) + o(1)\,.
$$
Here, $b_\star$ is a nonzero constant, and the $o(1)$ term tends to $0$ as $s \to 0$. 
\end{theorem}

\begin{remark} We stress that our method is easily generalizable to other spaces, by bifurcating from the corresponding geodesic balls in different spaces, as long as the space is not flat (see Subsection \ref{subsec:strategy} for a more detailed discussion on the role of the curvature). Our approach just requires a transversal intersection of a zonal Neumann eigenvalue and a non-zonal Dirichlet one. We numerically see these intersections for $\mathbb{S}^3$, $\mathbb{S}^4$, or the hyperbolic plane $\mathbb{H}^2$, as shown in Figure \ref{fig:appendix}. The proof would follow in an analogous way.
\end{remark}

Let us now turn to the study of Serrin-type overdetermined boundary value problems. More precisely, for $f \in C^1(\R)$, we consider overdetermined problems of the form
\begin{equation} \label{E.OverdeterminedSerrin}
\left\{
\begin{aligned}
\ & \Delta u + f(u) = 0 && \textup{in } \Omega\,,\\
& u > 0 && \textup{in } \Omega\,, \\
& u = 0 \,, \quad \partial_{\nu} u  = {\rm constant}&& \textup{on } \partial \Omega\,.
\end{aligned}
\right.
\end{equation}
The available literature for this type of problems is huge, and providing a systematic review goes beyond the scope of this paper. We will nevertheless highlight several results that are of interest to our paper. 

First, we discuss some classical rigidity results for overdetermined problems of this form. The seminal result by Serrin \cite{Serrin} states that if $\Omega \subset \R^n$, $n \geq 2$, is a $C^2$ bounded domain such that \eqref{E.OverdeterminedSerrin} admits a solution $u$, then $u$ is radially symmetric and the domain $\Omega$ must necessarily be a ball. A complete counterpart of this result was obtained by Kumaresan and Prajapat \cite{KP} for bounded domains $\Omega \subset \mathbb{H}^n$, and for domains $\Omega \subset \mathbb{S}^n$ with $\overline{\Omega} \subset \mathbb{S}^n_+$. 

Next, concerning flexibility results for overdetermined problems of the form \eqref{E.OverdeterminedSerrin}, we focus on the cases where $\Omega \subset \R^n$, $n \geq 2$, and where $\Omega \subset \mathbb{S}^n$, $n \geq 2$, and discuss them separately. 

In Euclidean space $\R^n$, Sicbaldi \cite{S10} first showed that the boundedness assumption cannot be removed from Serrin's rigidity result (see also \cite{FMW2, SS, RRS20} in this direction). On the other hand, Ruiz \cite{R} has recently shown that you cannot remove the positivity assumption either (see \cite{W} for another flexibility result without the positivity assumption). Finally, if we consider non-simply connected domains, it is natural to wonder what happens if one relaxes the boundary conditions, allowing for the solution to be locally constant on the boundary. Kamburov and Sciaraffia \cite{KS} constructed a parametric family of nontrivial annular domains, bifurcating from standard annuli, in which the relaxed version of \eqref{E.OverdeterminedSerrin} admits a solution (we refer as well to \cite{ABM} in this direction).  

In the case of the sphere $\mathbb{S}^n$, Fall, Minlend and Weth \cite{FMW3} first showed that the assumption $\overline{\Omega} \subset \mathbb{S}_+^n$ cannot be removed from Kumaresan and Prajapat's rigidity result (see also the very recent paper \cite{HKM}, where the authors prove flexibility results for \eqref{E.OverdeterminedSerrin} with $f(s) = \lambda_1(\Omega) s$, and analyze their connection with the one-phase free boundary problem). More recently, Ruiz, Sicbaldi and Wu \cite{RSW} proved the same for the more restrictive class of contractible domains. These results can be seen as the sphere counterparts of \cite{FMW2,RRS20}. At this point, it is very natural to wonder whether the counterpart of \cite{R,W} can be established in this setting or not. Our main result in this direction, which is an immediate corollary of Theorem \ref{T.intro}, provides an affirmative answer.

\begin{corollary} \label{C.intro}
    There exists a parametric family of smooth contractible domains $\Omega \subset \mathbb{S}^2$, with $\overline{\Omega} \subset \mathbb{S}_+^2$, such that
    the overdetermined problem
\begin{equation}  \label{E.OverdeterminedCorollaryIntro}
\left\{
\begin{aligned}
\ & \Delta v + \mu (v+1) = 0 && \textup{in } \Omega\,,\\
& v = 0 \,, \quad \partial_{\nu} v = 0 && \textup{on } \partial \Omega\,.
\end{aligned}
\right.
\end{equation}
    admits, for some $\mu > 0$, a nontrivial sign-changing solution. Here, $\Delta$ denotes the Laplace-Beltrami operator on $\mathbb{S}^2$ with respect to the standard round metric, and the family of domains $\Omega \equiv \Omega^{b_s}$ is the one given in Theorem \ref{T.intro}.
\end{corollary}

Before discussing the proof of our main result, namely Theorem \ref{T.intro}, we would like to highlight that Theorem \ref{T.intro} is the first flexibility result for Schiffer-type problems, which falls outside the flexibility setting for Serrin-type problems. More precisely, the flexibility results in \cite{FMW} can be somehow related to the ones in \cite{FMW2, FMW3}. Likewise, the flexibility result in \cite{EFRS} can be connected to \cite{KS,ABM}. Theorem \ref{T.intro} is the first flexibility result for a Schiffer-type problem which does not have a counterpart for Serrin-type problems.

\subsection{Computer-assisted techniques}

One of the main tools in our proof of Theorem \ref{T.intro} is the use of computer-assistance in the form of rigorous interval arithmetics in order to perform several computations. We will give an extremely short overview of this subject here. We remark the pioneering work of R. E. Moore on this subject \cite{Moore} and refer to \cite{Tucker} for an introduction to this topic. We also mention \cite{GS, NPW} as survey references for the specific application of rigorous interval arithmetics to PDE. 

For a generic operation $\ast$, and two intervals $x^I$, $y^I$, we can define the corresponding arithmetic on intervals by $$ x^I \ast y^I := \{ x \ast y \; : \;  x \in x^I, \; y \in y^I \}\,.$$
The main advantage of such interval arithmetic is that, in contrast with the element-wise one, it allows for rigorous implementations on the computer using finite precision. The reason behind this is that one can give up on computing the exact resulting interval, and simply ensure that the computed interval obtained from the computer \textit{encloses} the theoretical interval-arithmetic result. That property suffices to ensure rigor. For example, the sum of the intervals $x^I = [x^l, x^u]$ and $y^I = [y^l, y^u]$ is implemented as follows:
$$ x^I + y^I = [(x^l + y^l)^-, (x^u + y^u)^+]\,,$$ where we denote by $z^-$ (resp. $z^+$) the rounding-down (resp. rounding-up) operator that sends $z$ to the nearest lower (resp. upper) representable number by the computer. The crucial property of that definition is that \textit{for any instance of $x \in x^I$ and $y\in y^I$ we will have a rigorous guarantee that $x + y \in x^I + y^I$.} 

Such process is completely rigorous and independent of the architecture of the computer. In our case, we write our code in Sage, using arbitrary precision real balls to represent our intervals. The arithmetic is based on the arb module of FLINT. Our code can be found in the supplementary material to this paper (see TeX source on the arXiv version).

The use of computer-assisted techniques has traditionally allowed to solve extremely challenging problems in PDE and geometry, such as the relativistic stability of matter \cite{FdL}, the existence of Lorenz attractors \cite{Tucker:Lorenz}, or the Kepler conjecture \cite{Hales}. We stress that these techniques are gaining a lot of traction in recent years, with groundbreaking works such as the existence of singularities from $C^\infty$ initial data to the incompressible Euler equations with boundary \cite{CH2}, the existence of Einstein metrics on $\mathbb S^{12}$ \cite{BH}, the existence of self-similar profiles for the cubic focusing NLS equation \cite{DS}, or the solution to P\'olya's conjecture on balls \cite{FLPS}.

In the specific context of bifurcation problems, computer-assistance has also been used successfully. We remark \cite{CCG}, where the Crandall-Rabinowitz theorem was used in combination with computer-assistance to show the existence of global solutions for the inviscid SQG equation, and \cite{AK}, where the computer is used to prove properties of the bifurcation diagram for the Kuramoto–Sivashinski equation.

Computer-assisted techniques have also been used to understand spectral properties of the Laplacian. We point out again the recent groundbreaking work \cite{FLPS}, where computer-assistance is used to understand the spectral properties of the Neumann Laplacian on the ball, and resolve P\'olya's conjecture on that domain. In the case of the planar Dirichlet Laplacian, computer-assistance has been used to disprove that the first, second and fourth eigenvalues determine a triangle \cite{GO}, and to construct new and topologically simpler examples where the nodal line of the second eigenfunction does not touch the boundary of the domain \cite{DGH}. In our specific setting of understanding the Laplace--Beltrami operator on the sphere, we remark \cite{DahneSalvy}, where computer-assisted techniques were used to obtain rigorous enclosures for the Dirichlet eigenvalues of the Laplacian over spherical triangles.

\subsection{Strategy of the proof of Theorem \ref{T.intro}} \label{subsec:strategy}

Our approach to show Theorem \ref{T.intro} is to apply local bifurcation theory along the family of zonal solutions that exist over spherical caps of $\mathbb S^2$. We take $(\theta, \phi)$ to be our azimuthal and longitudinal angles and define our spherical caps by
\begin{equation} \label{E.spherical_cap}
\Omega_a:= \big\{(\sin(\theta)\cos(\phi), \, \sin (\theta )\sin(\phi),\, \cos(\theta)): \cos(\theta) > a \big\} \subset \mathbb{S}^2\,,
\end{equation}
which are the geodesic balls of $\mathbb S^2$ centered around the north pole (illustrated in blue on Figure \ref{fig:sphere}). Since those domains are symmetric around the $z$-axis, any zonal Neumann eigenfunction $u_0(\theta)$ will be constant on $\partial \Omega_a$.

\begin{figure}[t]
    \centering
    \includegraphics[width=0.5\textwidth]{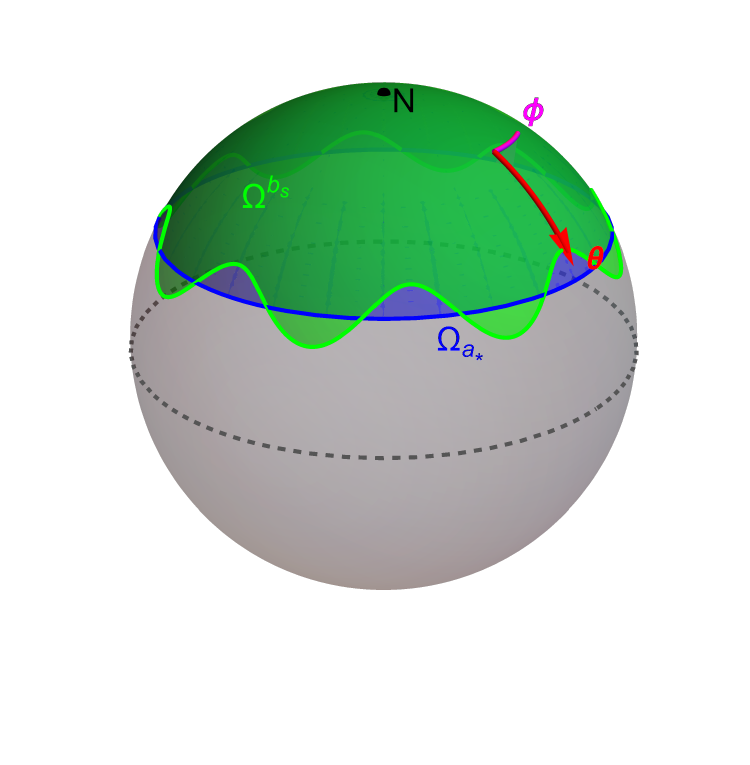} \hspace{1cm}
    \caption{Schematic picture of the spherical cap $\Omega_{a_\star}$ (blue) and our perturbed domain $\Omega^{b_s}$ (green). Our spherical coordinates are $(\theta, \phi)$ and our caps are centered around the north pole $\rm N $.}
    \label{fig:sphere}
\end{figure}

In order to obtain non-zonal solutions, we consider a perturbation of the form
$$
u(\theta, \phi) = u_0(\theta) + s u_1( \theta, \phi) + o(s)\,, \quad 0< |s | \ll 1\,,
$$
and perturb our domains $\Omega_a$ as well. The linearization of the overdetermined problem will impose that $u_1$ must be a Dirichlet eigenfunction of $\Omega_a$ with the same eigenvalue as $u_0$. Hence, the bifurcation condition becomes finding values $(a_\star, \lambda_\star)$ such that $\lambda_{\star}$ is both a zonal Neumann eigenvalue of $\Omega_{a_\star}$ (with eigenfunction $u_0$) and a non-zonal Dirichlet eigenvalue of $\Omega_{a_\star}$ (with eigenfunction $u_1$). In Figure \ref{fig:far}, we see a plot of the zonal Neumann spectrum in terms of $(a, \lambda)$, together with the Dirichlet spectrum over modes $\cos (8 \phi )$. Let us also point out the analogous plot in Figure \ref{fig:a} for multiples of $\cos (6 \phi )$.

\begin{figure}[t]
    \centering
    \includegraphics[width=0.465\textwidth]{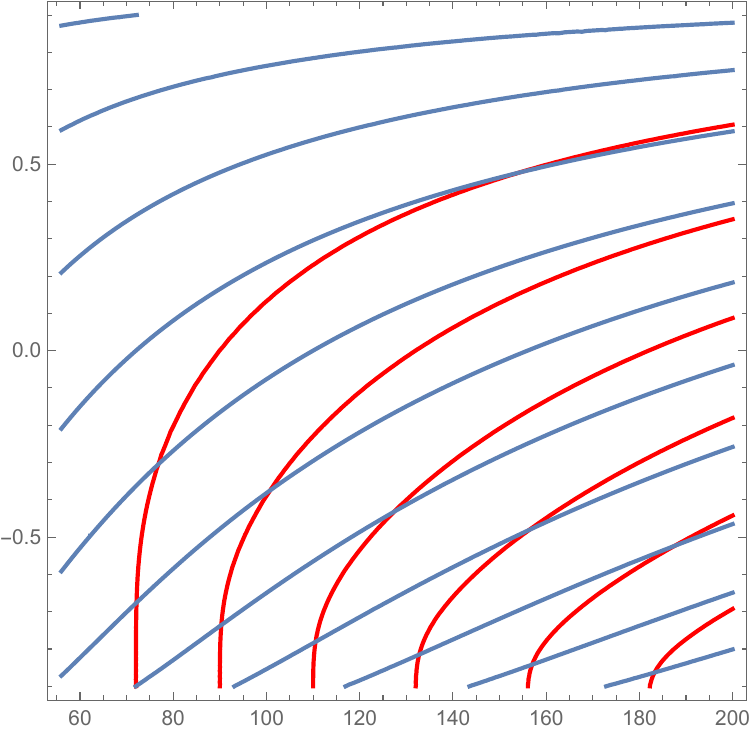} \hspace{1cm}
    \includegraphics[width=0.45\textwidth]{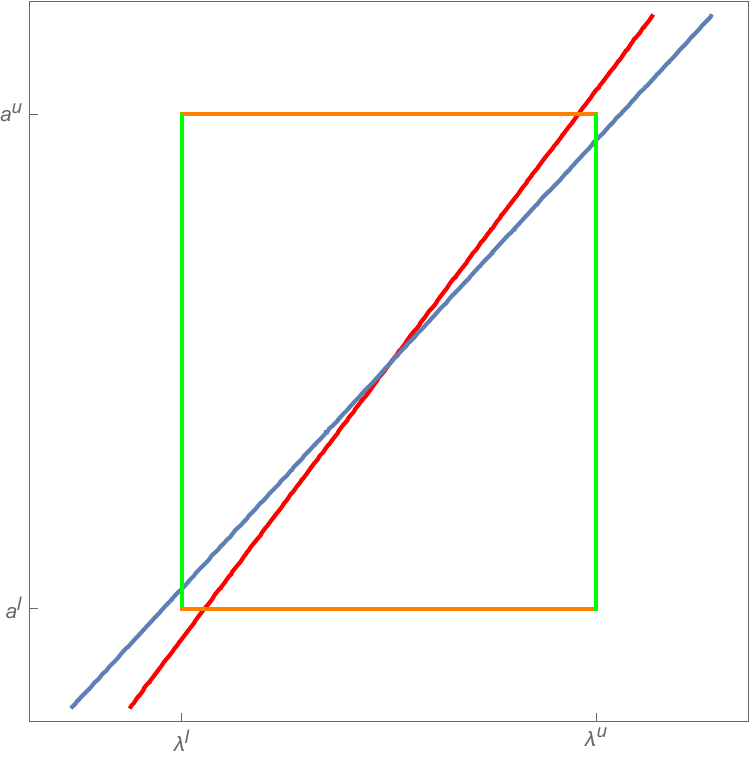}
    \caption{The plot shows the curves $(\lambda, a)$ for which $\lambda$ is an eigenvalue of $\Omega_a$. Blue corresponds to the cases where $\lambda$ is the eigenvalue of a zonal Neumann eigenfunction. Red corresponds to the cases where $\lambda$ is the eigenvalue of a Dirichlet eigenfunction with mode $\ell = 8$. We observe that at $a = -1$ we recover the sphere spectrum $n(n+1)$. Our intersection of interest happens at $a = 0.477\ldots$ and $\lambda = 154.19\ldots$, and we can visually see it is transversal on the right plot. The values of $\lambda^l$, $\lambda^u$, $a^l$, $a^u$ are given in \eqref{eq:rho_lambda_box}--\eqref{eq:a_nu_box}.}
    \label{fig:far} \label{fig:close}
\end{figure}

We focus our interest on the intersection $(a_\star, \lambda_\star) \approx (0.477 \ldots,\, 154.19\ldots)$ that we see in Figure \ref{fig:far}. We prove the existence of a non-zonal perturbation by using the Crandall-Rabinowitz theorem \cite{CR}. We structure our proof in two parts: (i) checking the properties of the spectrum that are needed for the perturbation argument;  (ii) implementing a suitable functional setting and checking the Crandall-Rabinowitz theorem hypotheses. Part (i) corresponds to Proposition \ref{P.ComputerAssisted}, which is proved in Section \ref{S.CAP}, and part (ii) corresponds to the proof of Theorem \ref{T.bifurcation} in Section \ref{S.bifurcation}.

Part (i) fundamentally consists in showing the intersection of the blue and red curves on Figure \ref{fig:far} at $(a_\star, \lambda_\star)$, as well as showing that the intersection is transversal. In order to do that we use that the eigenfunctions of $\Delta_g$ on spherical caps are given by Legendre functions of the first kind, and rely on computer-assisted techniques to find properties of the spectrum. In order to show properties with the computer, the first step is to reduce these properties to a set of inequalities. For example, the intersection of eigenvalues that we can see in Figure \ref{fig:far} is proved by means of the Poincar\'e-Miranda theorem, which just requires us to show an alternation of signs on the orange-green box from that figure. See Theorem \ref{T.PM} for the precise statement of the $2d$ -- Poincar\'e-Miranda theorem. Similarly, the transversality of the intersection reduces to an inequality of the derivatives of the Legendre functions.

Once we have reduced the problem to a set of inequalities involving Legendre functions of the first kind (concretely, Lemma \ref{L.CAP}) we rely on the computer. We implement a \textit{fully rigorous, interval-arithmetic version} of the Legendre functions of the first kind and their derivatives. Our function accepts intervals for both $\lambda^I$ and $a^I$ and outputs a value of the Legendre function that is \textit{rigorously guaranteed} to contain the corresponding result for any $(\lambda, a) \in \lambda^I \times a^I$. The implementation is based on a Taylor series around the origin. We implement the first $K = 100$ terms of the Taylor series using interval arithmetics, and use the rigorous bound on the tail of the series provided by Lemma \ref{L.tail}.

Regarding part (ii), we study our problem in a fixed domain $\mathbb D$ where we can write a pulled-back problem using a diffeomorphism $\Phi : \mathbb D \to \Omega_a^{b}$ (see \eqref{E.Omab} for the definition of the deformed domain $\Omega_a^b$).  We can then encode both the information of $u$ and the domain $\Omega_a^{b}$ on a single function $v$ satisfying Dirichlet boundary conditions on $\pd \mathbb D$. Moreover, we fundamentally use anisotropic H\"older spaces with one extra derivative on the azimuthal direction to formulate and check the Crandall-Rabinowitz conditions, as done in \cite{FMW, EFRS}. That avoids a loss of derivatives that otherwise arises from tracking the evolution of the boundary $\pd\Omega_a^b$. One fundamental difference with \cite[Theorem 1.5]{FMW} and \cite{EFRS} is the fact that our domains are contractible. This requires a more careful analysis at the north pole in order to ensure regularity. We surpass this obstruction using a boundary localized version of the diffeomorphism $\Phi$, which is more adequate for the contractible setting.

Lastly, we want to compare Theorem \ref{T.intro} with the Schiffer conjecture (see page 1). The obstruction in order to use our approach to tackle the Schiffer conjecture (and hence the Pompeiu problem) is the presence of the scaling symmetry in the Euclidean setting. In $\mathbb R^2$, the Laplacian over any ball $B_a(0)$ has the same spectrum, rescaled by $1/a^2$. In particular, the version of Figure \ref{fig:far} on the plane would only have curves of the form $\lambda = C/a^2$ and we would not have any transversal intersection.\footnote{ Empirically, we can also see from Figures \ref{fig:far}, \ref{fig:a} how the angle of the intersections decreases as $a \nearrow 1$. This is because the geometry of $\Omega_a$ (appropriately rescaled) converges to a flat disk, as $a \nearrow 1$.} Hence, we see that the geometry plays a crucial role: the curvature of $\Omega_a$ kills the scaling symmetry and allows us to have a the non-trivial behavior seen in Figure \ref{fig:far}, with transversal intersections. Once the transversal intersection of eigenvalues is provided, our techniques are agnostic to the geometry, and one can prove the same results on higher-dimensional spheres or the hyperbolic plane using the intersections on Figure \ref{fig:appendix}.

\section{The Dirichlet and Neumann spectrum of a spherical cap} \label{S.CAP}

We recall our spherical cap definition from \eqref{E.spherical_cap}, where $a \in (0, 1)$ denotes the height of the cap, and $(\theta, \phi) \in [0, \pi] \times \TT$ are the azimuthal and longitudinal spherical coordinates, as illustrated in Figure \ref{fig:sphere}. In this section, we prove the auxiliary results about Neumann and Dirichlet eigenvalues on spherical caps that we will need in the proof of our main result.

We start by setting some notation. For all $\nu \geq 0$ and all nonnnegative integer $\ell$, we denote by $P_\nu^\ell$ the Legendre function of the first kind of degree $\nu$ and order $\ell$. Then, we denote by $\kappa_{\ell,k}$ the $k$-th positive $\nu$-zero of 
$$
P_\nu^\ell(a) = 0\,, \quad \textup{with} \quad \nu \not \in \{0,1,2, \ldots,\ell-1\}\,,
$$
and set $\lambda_{\ell,n}(a) = \kappa_{\ell,n+1}(\kappa_{\ell,n+1}+1)$. Note that throughout the paper $n$ will be a nonnegative integer. It is well-known that an orthogonal basis of $L^2(\Omega_a, \sin(\theta) d\theta d \phi)$ consisting of Dirichlet eigenfunctions of the spherical cap $\Omega_a$ is

\begin{equation}
    \big\{ P_{\kappa_{\ell,n+1}}^0(\cos(\theta)),\, P_{\kappa_{\ell,n+1}}^\ell(\cos(\theta)) \cos(\ell \phi),\, P_{\kappa_{\ell,n+1}}^\ell(\cos(\theta)) \sin(\ell\phi): \ell \geq 1, n \geq 0  \big\}\,,
\end{equation}
and that the Dirichlet spectrum of the spherical cap $\Omega_a$, counting multiplicities is $\{\lambda_{\ell,n}(a)\}_{\ell,n=0}^\infty$.  We will omit the dependence of the eigenvalues on $a$ when no confusion may arise.

Likewise, we set
$\nu_{0,0} = 0$ and, for $(\ell,k) \neq (0,0)$, we denote by $\nu_{\ell,k}$ the $k$-th positive $\nu$-zero of
$$
\frac{\dd}{\dd s}P_\nu^\ell(s)\, \bigg|_{s=a} = 0\,, \quad \textup{with} \quad \nu \not \in \{0,1,2, \ldots,\ell-1\}\,.
$$
Then, we set
$$
\mu_{0,n}(a) := \nu_{0,n} (\nu_{0,n}+1) \quad \textup{ and } \quad \mu_{\ell,n}(a):= \nu_{\ell,n+1}( \nu_{\ell,n+1}+1)\,. 
$$
An orthogonal basis of $L^2(\Omega_a, \sin(\theta) d\theta d \phi)$ consisting now of Neumann eigenfunctions of the spherical cap $\Omega_a$ is 
\begin{equation}
    \big\{ P_{\nu_{0,n}}^0(\cos(\theta)),\,P_{\nu_{\ell,n+1}}^\ell(\cos(\theta)) \cos(\ell \phi),\, P_{\nu_{\ell,n+1}}^\ell(\cos(\theta)) \sin(\ell\phi): \ell \geq 1, n \geq 0  \big\}\,.
\end{equation}
and the corresponding Neumann spectrum of $\Omega_a$ is $\{\mu_{\ell,n}(a)\}_{\ell,n=0}^\infty$.  We will also omit the dependence of the eigenvalues on $a$ when no confusion may arise.

Having this notation at hand, we can state the main result of this section. The rest of this section is devoted to prove it. We use the standard notation from interval arithmetics where an interval in decimal notation is denoted by a string of common digits followed by subscripts and superscripts denoting the lower-bound and upper-bound following digits. For example $123_4^5 = [1234, 1235]$ and $0.4242_{00}^{42} = [0.4242, 0.424242]$.

\begin{proposition} \label{P.ComputerAssisted}
    There exist $a_\star \in  0.4774365682415_2^9$, $\lambda_\star \in 154.1915744945_0^2$ and $m_\star \in \mathbb N$, with $m_\star \geq 4$, such that:
    \begin{align}
        &\circ\ \lambda_\star =  \mu_{0,m_\star}(a_\star) = \lambda_{8,0}(a_\star)  \,, \tag{B} \label{E.B} \\
        &\circ\ \mu_{0,m_\star}'(a_\star) \neq \lambda_{8,0}'(a_\star)\,, \tag{T} \label{E.T}\\
        &\circ\ \lambda_{8,0}(a_\star) \neq \lambda_{8m,n}(a_\star)\, \quad \textup{for all } (m,n) \neq(1,0)\,. \hspace{7cm}\tag{NR} \label{E.NR}
    \end{align}
\end{proposition}
\begin{remark}
Numerically, we observe that $m_\star = 4$. One could use the computer-assisted methods developed in the proof of $n_\star = 0$ to show this (see Subsection \ref{subsec:ProofCAP}). However, the computation becomes quite expensive and would make sense to implement a second derivative with respect to the parameter $\lambda$ in order to have a faster method. Since this is not relevant for our main results, we have not done this.
\end{remark}

\begin{remark} We remark that one can prove an analogous Proposition to \ref{P.ComputerAssisted} for $\ell = 6$,  $a_\star = 0.81084\ldots$, and $\lambda_\star = 264.79\ldots$. We can see that intersection in Figure \ref{fig:a}. This is an example of a ``flatter'' geometry since $a_\star$ is closer to $1$.
\end{remark}

The proof of Proposition \ref{P.ComputerAssisted} is computer-assisted, based on interval arithmetics. We will write the bifurcation condition as the equation
$$
\frac{\dd}{\dd s} P_\nu^0(s) \Big|_{s=a} = 0 = P_\nu^8(a)\,, \quad (a,\nu) \in (0,1) \times (0,\infty)\,,
$$
and show the existence of one such solution $(a_\star, \nu_\star)$ by means of the Poincar\'e-Miranda theorem. This setting reduces the equation to a set of inequalities, and hence it is amenable for computer-assisted implementation. The main part will consist on the implementation of rigorous, interval arithmetic versions of $P_{\nu}^\ell (a)$, together with its derivatives in $a$ and $\nu$. The transversality and non-resonance conditions are already inequalities and therefore amenable to a direct implementation by interval arithmetics.

The rest of this section is organized as follows. In Subsection \ref{subsec:legendre} we detail our implementation of the Legendre functions of the first kind and their derivatives. The approach is based on a explicit computation of their Taylor series together with rigorous bounds for the tail. We defer the proof of Proposition \ref{P.ComputerAssisted} to Subsection \ref{subsec:ProofCAP}.

\subsection{Rigorous implementation of the Legendre Functions} \label{subsec:legendre}

The objective of this subsection is to derive an explicit recursive formula for the Taylor series coefficients of the Legendre functions of the first kind, and to rigorously bound the remainder of the series. This will allow us to implement a rigorous, interval-arithmetic version of $P_\nu^\ell (a)$ and its derivatives. Our implementation takes intervals in both slots of the Legendre function and outputs an interval that is \textit{rigorusly} known to contain the result $P_\nu^\ell (a)$ for any possible $(a,\nu)$ on their respective intervals. This is achieved by computing the Taylor series of the Legendre function at the origin up to a finite order $K = 100$, (with standard interval arithmetics that keeps rigorous error bounds on each coefficient) and then adding the rigorous bound on the remainder of the series. Thus, the main results of this subsection are the recurrence \eqref{eq:recurrence_p}--\eqref{eq:recurrence_q}, which allows us to compute the first $K$ Taylor coefficients, and Lemma \ref{L.tail}, which allows us to bound the remainder.

In order to obtain the simplest possible form for the Legendre functions of the first kind, we have chosen to work with the coordinate $\rho = \tan (\theta / 2) = \frac{\sqrt{1 - a}}{\sqrt{1 + a}}$, which corresponds to applying a stereographic projection to project the sphere minus the south pole to the equatorial plane. We also work with $\lambda$ instead of $\nu$, and recall that they are related by $\lambda = \nu (\nu + 1)$ or $\nu = \frac{-1 + \sqrt{ 1 + 4\lambda}}{2}$. 

For the remainder of this subsection, we drop the explicit dependence on $\nu, \ell$ and denote by $P(\rho)$ a generic Legendre function of mode $\ell$, eigenvalue $\lambda = \nu (\nu + 1)$ and written in terms of the stereographic coordinate $\rho = \tan (\theta / 2)$. We also denote by $Q$ its derivative with respect to $\lambda$, that is $Q (\rho) = Q_{\nu}^\ell (\rho) := \frac{\dd\nu}{\dd\lambda} \frac{\dd}{\dd\nu} P_{\nu}^\ell (\rho )$. Hence, the Legendre function $P$ satisfies
\begin{equation} \label{eq:algiers}
\left( \p_{\rho}^2 + \frac{\p_\rho}{\rho} - \frac{\ell^2}{\rho^2} \right) P + \frac{4\lambda}{(1+\rho^2)^2} P = 0 \qquad \textup{in } (0,\infty)\,,
\end{equation}

We can desingularize the expression at $ \rho = 0$ by considering $P(\rho) = \rho^\ell \bar P(\rho)$. Then, the function $\bar P$ is even and analytic, and satisfies the ODE
\begin{equation} \label{eq:Pbar_ode}
(1+\rho^2)^2 \left( \bar{P}''(\rho) + \frac{2\ell+1}{\rho}\bar{P}'(\rho) \right) + 4\lambda \bar{P}(\rho) = 0 \quad \textup{in } \R\,.
\end{equation}
Since $\bar{P}$ is an even function, we seek a series solution in powers of $\rho^2$, namely
\[
\bar{P}(\rho) = 1 + \sum_{k=1}^{\infty} \frac{p_k }{k ( k + \ell ) } \rho^{2k}\,,
\]
for some $p_k$ just depending on the parameters $\lambda$ and $\ell$. The recurrence reads as
\begin{equation} \label{eq:recurrence_p}
p_{k+1} = -   2 p_k - \frac{\lambda}{k ( k + \ell )} p_k  - p_{k-1} \quad \text{for } k \geq 2\,, \quad \mbox{ and } \quad p_1 = -\lambda, \quad p_2 = 2\lambda + \frac{\lambda^2}{1+\ell} \,.
\end{equation}
Similarly, we can express $Q(\rho) = \rho^\ell \bar Q (\rho)  $ with 
$$
\bar Q(\rho) = \sum_{k=1}^\infty \frac{q_k}{k (k + \ell )} \rho^{2k}\,.
$$
The coefficients $q_k$ are simply the derivative with respect to $\lambda$ of $p_k$. Even more, the recurrence for $q_k$ can be obtained simply deriving with respect to $\lambda$ the one above. We get that
\begin{equation} \label{eq:recurrence_q}
q_{k+1} = - 2 q_k - \frac{\lambda}{k ( k  + \ell ) } q_k - \frac{1}{k (k + \ell ) } p_k - q_{k-1} \quad \text{for } k\geq 2\,, \quad  \text{ and } \quad q_1 = -1\,, \quad q_2 = 2 + \frac{2\lambda}{1+\ell}\,.
\end{equation}
Equations \eqref{eq:recurrence_p}--\eqref{eq:recurrence_q} allow us to compute the Taylor coefficients for $P$ and $Q$ up to any order, using interval arithmetics. The fundamental missing ingredient for a computer-assisted proof is a rigorous bound on the contribution of the tail of the series. The following lemma gives us this bound.

\begin{lemma} \label{L.tail} Fix $K \in \mathbb N $ and assume that $\gamma > 1$ satisfies the inequality
\begin{equation}\label{eq:condition}
\frac{(\gamma - 1)^2}{\gamma} \geq \frac{2\lambda}{(K+2) ( K + 2 + \ell ) }
\end{equation}
Fix also $C > 0$ so that we have the bounds
\begin{equation}
| p_{K+1} + p_K | \leq C \lambda\,, \quad | q_{K+1} + q_K | \leq C\,, \quad |p_{K+1}| \leq \frac{C \lambda}{\gamma - 1}\,, \quad | q_{K+1}| \leq \frac{C}{\gamma - 1}\,.
\end{equation}
Then, we have that the series $\bar{P}(\rho) = 1 + \sum_{k=1}^{\infty} \frac{p_k }{k ( k + \ell ) } \rho^{2k}$ and $\bar Q(\rho) = \sum_{k = 1}^\infty \frac{q_k}{k ( k + \ell )} \rho^{2k}$ are absolutely convergent for $\rho^2 \gamma < 1$, and moreover we can estimate the tails for $P$ and $Q$ as follows:
\begin{align*}
\left| P(\rho) - 1 - \sum_{k = 1}^K \frac{p_k}{k ( k + \ell ) }\rho^{2k+\ell}  \right| &\leq \frac{C \lambda \rho^{\ell + 2K} \cdot \rho^2 \gamma}{K ( K + \ell ) (\gamma - 1) ( 1 - \rho^2 \gamma )} \,,
\\
\left| P'(\rho) - \sum_{k = 1}^K \frac{(2k + \ell) p_k}{ k (k + \ell )}\rho^{2k+\ell - 1} \right| &\leq  \frac{2 \lambda C \rho^{\ell + 2K-1} \cdot \rho^2 \gamma }{K  (\gamma - 1) ( 1 - \rho^2 \gamma )}\,,
\\
\left| Q(\rho) - \sum_{k = 1}^K \frac{q_k}{k ( k + \ell ) }\rho^{2k+\ell} \right| &\leq \frac{C \rho^{\ell + 2K} \cdot \rho^2 \gamma }{K ( K + \ell ) (\gamma - 1) ( 1 - \rho^2 \gamma )} \,,
\\
\left| Q'(\rho) - \sum_{k = 1}^K \frac{ (2k + \ell )q_k}{ k(k + \ell )}\rho^{2k+\ell - 1} \right| &\leq  \frac{2 C \rho^{\ell + 2K-1} \cdot \rho^2 \gamma }{K  (\gamma - 1) ( 1 - \rho^2 \gamma )}\,.
\end{align*}
\end{lemma}

\begin{proof}
The proof is based on the propagation by induction of the bounds:
\begin{align} \label{eq:induction1}
|p_{K+1+j}| &\leq \frac{C \gamma \lambda}{\gamma - 1} \gamma^j\,, \qquad  |p_{K+1+j} + p_{K+j}| \leq C \lambda \gamma^j\,, \\
\label{eq:induction2}
|q_{K+1+j}| &\leq \frac{C \gamma}{\gamma - 1}\gamma^j\,,  \qquad |q_{K+1+j} + q_{K+j}|  \leq C \gamma^j \,.
\end{align}
At $j = 0$ those bounds are true by hypothesis. The bounds on the left for $p_{K+1+j}$ and $q_{K+1+j}$ trivially follow from the ones on the right and the previous induction step, using that $\frac{C \gamma }{\gamma - 1} \gamma^{j-1} + C\gamma^j = C \gamma^j \frac{\gamma}{\gamma - 1}$. Hence, we just need to show the bounds on the right using the induction hypothesis. We use the recurrences \eqref{eq:recurrence_p}--\eqref{eq:recurrence_q} to bound:
\begin{align*}
|p_{k+1} + p_k | &\leq | p_{k} + p_{k-1} | + \frac{\lambda }{k ( k + \ell )} |p_k|\,, \\
|q_{k+1} + q_k | &\leq | q_{k} + q_{k-1} | + \frac{\lambda }{k ( k + \ell )} |q_k| + \frac{1}{k (k + \ell ) } | p_k |\,.
\end{align*}
Now, we take $k = K + j$ and use the induction hypothesis to bound the right hand side:
\begin{align*}
|p_{K+j+1} + p_{K+j} | &\leq C \lambda \gamma^{j-1} + \frac{\lambda }{(K+j) ( K + j + \ell )} \frac{C \gamma \lambda}{\gamma - 1} \gamma^{j-1} \leq C \lambda \gamma^{j-1} \left( 1 + \frac{\lambda}{K (K+\ell)} \cdot \frac{\gamma}{\gamma - 1} \right)\,, \\
|q_{K+j+1} + q_{K+j} | &\leq C  \gamma^{j-1} + \frac{\lambda }{(K+j) ( K + j + \ell )} \frac{C \gamma}{\gamma - 1}\gamma^{j-1} + \frac{1}{(K+j) (K + j + \ell ) } \frac{C \lambda \gamma}{\gamma - 1} \gamma^{j-1} \\
&\leq C \gamma^{j-1} \left( 1 + \frac{\lambda }{K( K+ \ell )}  \frac{\gamma }{\gamma - 1} + \frac{\lambda}{K (K + \ell ) } \frac{\gamma }{\gamma - 1} \right)\,.
\end{align*}

Condition \eqref{eq:condition} tells us that $1 + \frac{2\lambda}{K(K+\ell)} \frac{\gamma }{\gamma - 1} \leq \gamma$, and hence, we conclude that 
\begin{align*}
|p_{K+j+1} + p_{K+j}| \leq C\lambda \gamma^j \quad \textup{ and } \quad |q_{K+j+1} + q_{K+j}| \leq C \gamma^j\,.
\end{align*}
Therefore, we close our induction argument, and this ends the proof of \eqref{eq:induction1}--\eqref{eq:induction2}. 

Moreover, from our bounds on $p_k,\ q_k$, and comparison with $\sum_k 1/k^2$, it is clear that the series converges uniformly on $\rho \leq 1/ \sqrt{\gamma}$. 

Finally, we derive our bounds for the tails. We directly notice that 
\begin{align*}
\left| \sum_{k = K+1}^\infty \frac{q_k}{k (k + \ell ) } \rho^{2k+\ell} \right| 
&\leq \frac{\rho^{\ell + 2K}}{K ( K + \ell ) } \, \frac{C}{\gamma - 1} \sum_{j = 1}^\infty  \gamma^{j} \rho^{2j}  \leq \frac{\rho^{\ell + 2K}}{K ( K + \ell ) } \, \frac{C}{\gamma - 1} \, \frac{\rho^2 \gamma}{  1 - \rho^2 \gamma }\,, \\
%
%
\left| \sum_{k = K+1}^\infty \frac{q_k (2k + \ell ) }{k (k + \ell ) } \rho^{2k + \ell - 1} \right| 
&\leq \frac{2 \rho^{\ell + 2K - 1}}{K} \cdot \frac{C}{\gamma - 1}  \sum_{j = 1}^\infty \gamma^{j} \rho^{2j} \leq \frac{2 C \rho^{\ell + 2K-1} }{K  (\gamma - 1) } \frac{\rho^2 \gamma}{1 - \rho^2 \gamma}.
\end{align*}
The bounds for $P$ follow in the exact same way, but carrying the extra factor of $\lambda$ from the bounds \eqref{eq:induction1}--\eqref{eq:induction2}
\end{proof}

\subsection{Proof of Proposition \ref{P.ComputerAssisted}} \label{subsec:ProofCAP}
Since our Legendre functions are implemented in terms of $\rho$ and $\lambda$ (instead of $a$ and $\nu$), we will work on those variables. To that end, we denote with a double subindex the Legendre function in $\lambda$-$\rho$ variables, that is 
$$ P_{\lambda, \ell} (\rho) := P_{-\frac12 + \frac12 \sqrt{1+\lambda}}^\ell \left( \frac{1-\rho^2}{1+\rho^2}\right)\,,
$$ 
where we recall that $\nu = -\frac12 + \frac12 \sqrt{1 + \lambda}$ (so that $\lambda = \nu (\nu+1)$) and $\cos (\theta) = \frac{1-\rho^2}{1+\rho^2}$ given that $\rho = \tan (\theta / 2)$.

We will be working on the box
\begin{equation} \label{eq:rho_lambda_box}
 \rho^I = [\rho^l, \rho^u] := 0.5947234806949_{31}^{70}, \qquad \lambda^I = [\lambda^l, \lambda^u] := 154.1915744945_{05}^{20}
 \end{equation}
It is a standard interval arithmetic computation to check that for any $\rho \in \rho^I$ and any $\lambda \in \lambda^I$ we have that $a = \frac{1-\rho^2}{1+\rho^2}$ and $\nu = -\frac12 +  \sqrt{\frac14 + \lambda}$ lie in the following intervals:
\begin{equation} \label{eq:a_nu_box}
a^I = [a^l, a^u] = 0.4774365682415_2^9, \qquad \nu^I = [\nu^l, \nu^u] = 11.92745245392_{25}^{32}
\end{equation}
We organize the proof as follows. The core of the proof is contained in the following Lemma, which is proved using an interval arithemtics computer assisted proof. The rest of the section explains how to conclude Proposition \ref{P.ComputerAssisted} from Lemma \ref{L.CAP}

\begin{lemma} The following items hold: \label{L.CAP} 
\begin{itemize}
\item[(i)] \hspace{-0.15cm} -- {\rm(CAP-B) } We have the inequalities
\begin{align} \label{eq:PM1}
P_{\lambda^l, 8} (\rho^u) > 0\,, \qquad P_{\lambda^u, 8} (\rho^l) < 0\,, \qquad \frac{\dd}{\dd \rho} P_{\lambda, 8} (\rho) < 0\,,\;\; \forall\ (\lambda, \rho) \in \lambda^I \times \rho^I\,, \\ \label{eq:PM2}
P_{\lambda^l, 0}' (\rho^u) < 0\,, \qquad P_{\lambda^u, 0}' (\rho^l) > 0\,, \qquad \frac{\dd}{\dd \lambda} P_{\lambda, 0}' (\rho) < 0, \;\; \forall\ (\lambda, \rho) \in \lambda^I \times \rho^I\,.
\end{align}
\item[(ii)] \hspace{-0.15cm} -- {\rm (CAP-AD) } Let $\lambda^{\rm{aux}} := \frac{1541914}{10000}$. Then:
\begin{itemize}
    \item[$\circ$] For any $\lambda \in [2, \lambda^{\rm{aux}}]$ and every $\rho \in \rho^I$, we have that $P_{\lambda, 8} (\rho) >  0\,.$
    \item[$\circ$]  For any $\lambda \in [\lambda^{\rm{aux}}, \lambda^u]$ and every $\rho \in \rho^I$ we have that $Q_{\lambda, 8} (\rho) < 0\,.$
\end{itemize}

\item[(iii)] \hspace{-0.15cm} -- {\rm (CAP--LN) } For every $\rho \in \rho^I$, it follows that:
\begin{align*}
P_{12, 0}' (\rho) < 0 < P_{13, 0}' (\rho)\,, \qquad 
P_{42, 0}' (\rho) > 0 > P_{43, 0}' (\rho)\,, \qquad
P_{89, 0}' (\rho) < 0 < P_{90, 0}' (\rho)\,. 
\end{align*}
\item[(iv)] \hspace{-0.15cm} -- {\rm (CAP-T) } For every $(\rho, \lambda) \in \rho^I \times \lambda^I$, we have
\begin{equation} \label{eq:paris}
\ \frac{4 \lambda}{(1 + \rho^2)^2} P_{\lambda, 0} ( \rho ) Q_{\lambda, 8} (\rho) + P_{\lambda, 8}' (\rho ) Q_{\lambda, 0}' (\rho ) \neq 0\,, \qquad Q_{\lambda, 8} (\rho ) \neq 0\,, \qquad Q_{\lambda, 0}' (\rho ) \neq 0 \,.
\end{equation}
\item[(v)] \hspace{-0.15cm} -- {\rm (CAP-NR) }  For every $(\rho, \lambda) \in \rho^I \times \lambda^I$, it follows that $ P_{\lambda, 0} (\rho) \neq 0$.
\end{itemize}
\end{lemma}

\begin{proof}
The proof is computer-assisted, and the code is in file \url{CAP\_V6\_SubmissionVersion.ipynb} in the TeX source of our arXiv version. The main part of the code is the rigorous implementation of interval arithmetic version of the Legendre functions of the first kind and their derivatives with respect to $\lambda$ and $\rho$. We compute the first $K = 100$ terms of the series symbolically using \eqref{eq:recurrence_p}--\eqref{eq:recurrence_q}. We store $p_n$, $q_n$ as polynomials on $\lambda$ of degree $n$. Then, the implementation is based on Lemma \ref{L.tail}, using standard interval arithmetics to compute the truncated series and adding the upper bound on the tail contribution given by Lemma \ref{L.tail}. We set $\gamma = 3/2$ and we check condition \eqref{eq:condition} and $\rho^2 \gamma < 1$ on every call to our Legendre function. Thus, we end up with an rigorous implementation of $P_{\lambda^V, \ell} (\rho^V)$, $Q_{\lambda^V, \ell} (\rho^V)$ and their first derivative with respect to $\rho$. The functions accept intervals for $\lambda^V$ and $\rho^V$ and return an interval \textit{which is rigorously guaranteed to contain the result} for any $\rho \in \rho^V$ and any $\lambda \in \lambda^V$.

Then, the proof of all items except the first subitem of $(ii)$ becomes immediate. All of them are direct applications of the implementation described above, checking that the resulting interval has the corresponding sign. Whenever one of the inputs is a number (such as $\rho^l$ or $\rho^u$) we simply create an interval which contains its actual value and use our interval-arithmetic implementation of the corresponding function.

Regarding the first subitem of $(ii)$, the interval $[2, \lambda^{\rm{aux}}]$ is too large to obtain a sign if we simply evaluate $P_{[2, \lambda^{\rm{aux}}], 8} (\rho^I)$. In order to tackle that interval, we exploit our implementation of the derivative with respect to $\lambda$. Whenever we want to check if $P_{\lambda^V, 8} (\rho^I)$ contains zero for some interval $\lambda^V = [\lambda^-, \lambda^+]$, we operate as follows. We define $\lambda^m$ containing $\frac{\lambda^- + \lambda^+}{2}$, we define $r > \frac{\lambda^+ - \lambda^-}{2}$, and we use our Legendre functions implementations to compute $\alpha \supset P_{\lambda^m, 8}(\rho^I)$, and $\beta \supset  Q_{\lambda^V,8} (\rho^I)$. By Intermediate Value theorem, if there is some $\lambda \in \lambda^V$ and $\rho \in \rho^I$ with $P_{\lambda, 8}(\rho) = 0$, we would have $\lambda' \in \lambda^V$ such that $P_{\lambda^m, 8} (\rho) + (\lambda - \lambda^m)  Q_{\lambda',8}(\rho) = 0$. Hence, in order to show that there are no zeros in the interval, it suffices to check that $0 \notin \alpha + [-r, r] \cdot \beta$.

Such improvement is still far from dealing with the whole interval $[2, \lambda^{\text{aux}}]$. We combine it with a standard branch-and-bound method to check the condition. For each interval $\lambda^V$, we try to check the condition above, and, if we fail, we split the interval in two halves and check each. The program stops when the size of the interval reaches a tolerance of $10^{-6}$. For our problem, such tolerance is never reached.

All items but the first subitem of $(ii)$ take a few seconds in a standard laptop CPU, while the first part of $(ii)$ takes 126 seconds. We remark that it would be fairly straightforward to make this verification much faster by implementing rigorous higher order derivatives of $P_\lambda^\ell(\rho)$ with respect to $\lambda$ and using Taylor's Theorem as a substitute of our Intermediate Value Theorem. However, since our computing times are already very small we opted to keep a simpler code, and a simpler version of Lemma \ref{L.tail} instead of optimizing code performance.
\end{proof}

\subsubsection{Proof of \texorpdfstring{\eqref{E.B}}{E.B}}

We start proving condition \eqref{E.B} from Proposition \ref{P.ComputerAssisted}. As explained in the introduction, the main ingredient is the application of Poincar\'{e}-Miranda theorem \cite{Miranda}.

\begin{theorem}[$2d$ -- Poincar\'e-Miranda Theorem] \label{T.PM} Let $x^l, x^u, y^l, y^u \in \mathbb R$ and define the intervals $I = [x^l, x^u]$, $J = [y^l, y^u]$. Assume that $f_1, f_2 : I \times J \to \mathbb R^2$ are continuous, and that there exist sign choices $\sigma_1, \sigma_2 \in \{ -1, +1 \}$:
\begin{itemize}
\item[$\circ$] $\textup{sign}(f_1 (x^l, y)) = \sigma_1$ and $\textup{sign}(f_1 (x^u, y)) = -\sigma_1\,$ for all $y \in J\,,$
\item[$\circ$] $\textup{sign}(f_2(x, y^l )) = \sigma_2$ and $\textup{sign}(f_2(x, y^u )) = -\sigma_2\,$ for all $x\in I\,.$
\end{itemize}
Then, there exists $(x, y) \in I\times J$ such that $f_1 (x, y) = f_2 (x, y) = 0$. 
\end{theorem}
 Looking at Figure \ref{fig:close} the condition can be written in words as asking that $f_1$ alternates sign between the two orange segments, and $f_2$ alternates sign between the green segments.

Now, we will see that \eqref{eq:PM1}--\eqref{eq:PM2} allow us to apply Poincar\'e-Miranda. Specifically, from \eqref{eq:PM1}, we obtain:
\begin{equation} \label{eq:palau1}
P_{\lambda^l, 8} (\rho) > 0, \qquad P_{\lambda^u, 8}(\rho) < 0, \qquad \forall\ \rho \in \rho^I\,.
\end{equation}
Similarly, from \eqref{eq:PM2}, we deduce
\begin{equation} \label{eq:palau2}
P_{\lambda, 0}' (\rho^u) < 0, \qquad P_{\lambda, 0}'(\rho^l) > 0, \qquad \forall\ \lambda \in \lambda^I\,.
\end{equation}
Equations \eqref{eq:palau1}--\eqref{eq:palau2} are exactly the conditions needed to apply Poincar\'e-Miranda theorem. We conclude the existence of $(\rho_\star, \lambda_\star) \in \rho^I \times \lambda^I$ solving the system
\begin{equation}
P_{\lambda_\star, 8}(\rho_\star) = 0, \qquad \mbox{ and } \qquad P_{\lambda_\star, 0}'(\rho_\star) = 0.
\end{equation}
Lastly, we express this in terms of $a_\star = \frac{1-\rho_\star^2}{1+\rho_\star^2} \in a^I$ and $\nu_\star = -\frac12 + \frac12 \sqrt{1 + \lambda_\star} \in \nu^I$. From our discussion at the beginning of the section, this means that $\lambda_\star = \nu_\star (\nu_\star + 1)$ is both a zonal Neumann eigenvalue of $\Omega_a$ and a Dirichlet $8$-fold eigenvalue. That is, we conclude that there exists $a_\star \in [a_1, a_2]$, $m_\star, n_\star \in \mathbb N$ such that
\begin{equation}
\mu_{0, m_\star}(a_\star) = \lambda_{8, n_\star}(a_\star) \in \lambda^I\,.
\end{equation}

This concludes the proof of first item in Proposition \ref{P.ComputerAssisted}, except for the facts that $n_\star = 0$ and $m_\star \geq 4$, which we prove in the following.

We start showing $n_\star = 0$. First, we notice that $\lambda_{8, 0}(a_\star) > \lambda_{0, 0}(a_\star)$, where $\lambda_{0, 0}(a_\star)$ is the smallest eigenvalue of the Dirichlet Laplacian on $\Omega_{a_\star}$. From the variational characterization of the first eigenfunction of the Laplacian, and the fact that $\Omega_{a_\star} \subset \Omega_0 = \mathbb{S}^2_{+}$, we have that $\lambda_{0, 0}(a_\star) \geq \lambda_{0, 0}(0) = \lambda_1(\mathbb{S}_{+}^2)$. The first eigenvalue of the Dirichlet Laplacian on the half-sphere $\mathbb{S}^2_+$ is known to be $2$ (for example, see \cite[Proposition 4.5]{CM}). Hence, $\lambda_{8, 0}(a_\star) > 2$.

From the first part of item $(ii)$ of Lemma \ref{L.CAP}, we know that there are no Dirichlet eigenvalues with $\ell = 8$ and $\lambda \in [2, \lambda^{\rm{aux}}]$. Hence, $\lambda_{8, 0}(a_\star) \geq \lambda^{\rm{aux}}$. Lastly, from the second item of Lemma \ref{L.CAP} $(ii)$, combined with Rolle's Theorem, we see that $P_{\lambda, 8} (\rho_\star ) = 0$ can only happen at most once for $\lambda \in [\lambda^{\rm{aux}}, \lambda^u]$. Since it happens once (for $\lambda = \lambda_\star$), we conclude that $\lambda_\star$ is the first Dirichlet eigenvalue with $\ell = 8$, and hence $n_\star = 0$ and $\lambda_\star = \lambda_{8, 0}(a_\star)$.

Lastly, we show that $m_\star \geq 4$ using Lemma \ref{L.CAP} $(iii)$. By continuity of $P_{\lambda, 0}'(\rho_\star)$ with respect to $\lambda$ and the sign conditions from Lemma \ref{L.CAP} $(iii)$, we see that there is at least one Neumann zonal eigenvalue on each of the following intervals: $[12, 13]$, $[42, 43]$, $[89, 90]$. Hence, since $\lambda_\star \in [154, 155]$, we see that $\lambda_\star$ is at least the fourth zonal Neumann eigenvalue of $\Omega_{a_\star}$. Moreover, we recall that in our convention $\mu_{0, 0}(a_\star) = 0$ corresponding to the constant Neumann eigenfunction of eigenvalue $0$. Hence, we obtain $m_\star \geq 4$.

\subsubsection{Proof of \texorpdfstring{\eqref{E.T}}{E.T}}

Here, we prove the transversality condition, namely the second item in Proposition \ref{P.ComputerAssisted}. For any zonal Neumann eigenfunction, and taking $a = a(\rho): = \frac{1-\rho^2}{1+\rho^2}$, we have the boundary condition 
$$
P_{\mu_{0, m} (a) , 0}'(\rho) = 0\,.
$$ 
Assuming $Q_{\mu_{0, m_\star}(a), 0}'(\rho) \neq 0$, and taking derivative with respect to $\rho$, it follows that
\begin{equation*}
\frac{\dd a}{\dd\rho } \mu_{0, m_\star} '(a) Q_{\mu_{0, m_\star}(a), 0 }'(\rho) + P_{\mu_{0, m_\star}(a), 0}''(\rho) = 0\,,
\end{equation*}
and so that
\begin{equation*}
\mu_{0, m_\star}'(a) = \left( \frac{\dd a}{\dd \rho} \right)^{-1} \cdot \frac{- P_{\mu_{0, m_\star}(a), 0}''(\rho) }{Q_{\mu_{0, m_\star}(a), 0}'(\rho)}\,.
\end{equation*}
Moreover, we can use the ODE \eqref{eq:algiers} to express the second derivative of $P$ in terms of zeroth and first order derivatives, obtaining that
\begin{equation} \label{eq:berlin1} 
\mu_{0, m_\star}'(a) = \left( \frac{\dd a}{\dd \rho}\right)^{-1} \frac{ \frac{ P_{\mu_{0, m_\star} (a), 0}'(\rho) }{\rho} + \frac{4 \mu_{0, m_\star} (a)}{(1+\rho^2)^2} P_{\mu_{0, m_\star} (a) , 0} (\rho)  }{Q_{\mu_{0, m_\star} (a), 0}'(\rho)} =  \left( \frac{\dd a}{\dd\rho}\right)^{-1} \,\frac{4 \mu_{0, m_\star} (a) }{ (1+\rho^2)^2}\, \frac{ P_{\mu_{0, m_\star} (a) , 0} (\rho)  }{Q_{\mu_{0, m_\star} (a), 0}'(\rho)}\,.
\end{equation}
Similarly, using the boundary condition
$$
P_{\lambda_{8,0}(a),8}(\rho) = 0\,,
$$
we get that
\begin{equation} \label{eq:berlin2}
\lambda_{8,0} ' (a) = - \left( \frac{\dd a}{\dd\rho}\right)^{-1} \, \frac{ P_{\lambda_{8,0} (a), 8}'(\rho )}{Q_{\lambda_{8, 0} (a), 8}(\rho)}\,.
\end{equation}

For our case of $\mu_{0, m_\star}(a_\star) = \lambda_{8, 0}(a_\star) = \lambda_\star \in \lambda^I$ and $\rho = \rho_\star \in \rho^I$, we know that the denominators $Q_{\lambda_\star, 0}' (\rho_\star )$ and $Q_{\lambda_\star,8} (\rho_\star )$ do not vanish due to the second and third inequalities in Lemma \ref{L.CAP} $(iv)$. Moreover, using \eqref{eq:berlin1}--\eqref{eq:berlin2}, we conclude $\mu_{0, m_\star}' (a_\star) \neq \lambda_{8, 0}' (a_\star)$ directly from the first inequality in Lemma \ref{L.CAP} $(iv)$.

\subsubsection{Proof of \texorpdfstring{\eqref{E.NR}}{E.NR}}

The reduction of condition \eqref{E.NR} to the computer-assisted check given in Lemma \ref{L.CAP} $(v)$ follows from standard ordering properties of the eigenvalues for the Dirichlet Laplacian. We recall those properties of $\{\lambda_{\ell,n}\}_{\ell,n=0}^\infty$, which are a direct consequence of the Courant-Fischer min-max characterization of eigenvalues:

\begin{itemize}
    \item For a fixed angular mode $\ell$, the eigenvalues are strictly increasing with the index $n$: $\lambda_{\ell, n} < \lambda_{\ell, n+1}$. This is because the $(n+1)$-th eigenvalue is found by maximizing over a space of higher dimension than the $n$-th.
    \item For a fixed index $n$, the eigenvalues are strictly increasing with the angular mode $\ell$: $\lambda_{\ell, n} < \lambda_{\ell+1, n}$. This is because the energy functional being minimized (the Rayleigh quotient) contains a term proportional to $\ell^2$, which increases the minimum value for larger $\ell$.
\end{itemize}

Therefore, we have that
$$ \lambda_{8, 0}(a) \leq \lambda_{8j, 0}(a) \leq \lambda_{8j, n'}(a)$$
for any $j \geq 1$, $n' \geq 0$, and both equalities hold if and only if $(j, n') = (1, 0)$. This shows that $\lambda_{8, 0}(a) \neq \lambda_{\ell', n'}(a)$ for any $\ell' \in 8 \mathbb N$, $n' \in \mathbb N$, $(\ell', n') \neq (\ell, n)$, except for the cases $\ell' = 0$, which are handled separately.

In order to deal with the remaining case $\ell' = 0$, we just need to check that $\lambda_{8, 0}(a_\star) = \lambda_\star$ is not a zonal Dirichlet eigenvalue of $\Omega_{a_\star}$. In other words, we need to check that $P_{\lambda_\star, 0} (\rho_\star) \neq 0$, which is guaranteed by Lemma \ref{L.CAP} $(v)$.

\section{The bifurcation argument} \label{S.bifurcation}

In this section, assuming slightly more general conditions than \eqref{E.B}, \eqref{E.T} and \eqref{E.NR}, we are going to prove the existence of nontrivial simply connected domains $\Omega \subset \mathbb{S}^2$ for which \eqref{E.OverdeterminedM} admits a nontrivial solution. The proof of this result is independent of Proposition \ref{P.ComputerAssisted}, and it is not computer-assisted. 

Before going further, let us recall that we denote the Neumann spectrum of $\Omega_a$ by $\{\mu_{\ell,n}(a)\}_{\ell,n=0}^\infty$, and the Dirichlet spectrum by $\{\lambda_{\ell,n}(a)\}_{\ell,n=0}^\infty$. Also, for simplicity, we denote the zonal part of the corresponding Neumann and Dirichlet eigenfunctions (whose dependence on $a$ we omit notationally) by $\{\psi_{\ell,n}\}_{\ell,n=0}^\infty$ and $\{\varphi_{\ell,n}\}_{\ell,n=0}^\infty$, respectively. In other words, for nonnegative integers $\ell,m$ and $n$ with $\ell \geq 1$, we set 
\begin{equation} \label{E.eigenfunctionspsiphi}
 \psi_{0,n}(\theta) := P_{\nu_{0,n}, }^0(\cos(\theta))\,,\quad \psi_{\ell,n} (\theta):= P_{\nu_{\ell,n+1}}^\ell(\cos(\theta))\,, \quad \textup{and} \quad  \varphi_{m,n}(\theta) := P_{\kappa_{m,n+1}}^m(\cos(\theta))\,. 
\end{equation}

Also, we set $\widetilde{a}:= \arccos(a) \in (0,\pi/2)$ for all $a \in (0,1)$, and point out that in spherical coordinates
$$
\Omega_a = \big\{(\theta,\phi) \in [0,\pi] \times \TT: \theta < \widetilde{a} \big\}\,,
$$
with the usual identification where $\theta = 0$ denotes the north pole irrespectively of $\phi \in \TT$. Then, given a constant $a \in (0,1)$ and a function $b \in C^{2,\alpha}(\TT)$ satisfying
$$
\|b\|_{L^\infty(\TT)} < \frac{\ta}{10}\,,
$$
we consider the deformed domain defined in spherical coordinates given by
\begin{equation} \label{E.Omab}
\Omega_{a}^b := \big\{(\theta,\phi) \in [0,\pi] \times \TT : \theta < \widetilde{a}+b(\phi) \big\}\,.
\end{equation}

The rest of this section is devoted to prove the following:

\begin{theorem} \label{T.bifurcation}
    Let $\ell \geq 4$, $m \geq 1$ and $n \geq 0$. Assume that there exists $a_\star \in (0,1)$ such that:
    \begin{itemize}
        \item[$(i)$] $\mu_{0,m}(a_\star) = \lambda_{\ell,n}(a_\star)\quad$ (Bifurcation)
        \item[$(ii)$] $\mu'_{0,m}(a_\star) \neq \lambda'_{\ell,n}(a_\star)\quad $ (Transversality)
        \item[$(iii)$] $\lambda_{\ell,n}(a_\star) \neq \lambda_{\overline{m}\ell,\overline{n}}(a_\star)$ \textup{ for all nonnegative integers} $\overline{m}, \overline{n}$ with $(\overline{m},\overline{n}) \neq (1,n) \quad$ (Nonresonance)  
    \end{itemize}
Then, there exist some $s_\star > 0$ and a continuously differentiable curve
$$
\big\{(a_s, b_s): s \in (-s_\star,s_\star),\ (a_0,b_0) = (a_\star,0) \big\} \subset (0,1) \times C^{2,\alpha}(\TT)\,,
$$
with
\begin{equation} \label{E.bs}
b_s(\phi) = - s\,\, \frac{ \varphi_{\ell,n}'(\widetilde{a}_\star) }{\psi_{0,m}''(\widetilde{a}_s)}\, \cos(\ell \phi) + o(s)\,,
\end{equation}
such that the overdetermined problem
\begin{equation} \label{E.overdeterminedCR}
\left\{
\begin{aligned}
\ & \Delta u_s +\mu_{0,m}(a_s) u_s = 0 && \textup{in } \Omega_{a_s}^{b_s}\,,\\
& u_s = \psi_{0,m}(\ta_s)\,, \quad \partial_{\nu} u_s  = 0 \quad&& \textup{on } \partial \Omega_{a_s}^{b_s}\,,
\end{aligned}
\right.
\end{equation}
admits a nonconstant solution for every $s \in (-s_\star,s_\star)$. Moreover, both the solution $u_s$ and the boundary of the domain $\Omega_{a_s}^{b_s}$ are analytic. 
\end{theorem}

\begin{remark}
    In the remainder of this section we assume that $\ell \geq 4$, $m \geq 1$, and $n \geq 0$, are nonnegative integers. 
\end{remark}

\subsection{Set up and functional setting}

We want to show that there exist some constant $a \in (0,1)$ and some nontrivial function $b \in C^{2,\alpha}(\TT)$ for which there is a Neumann eigenfunction $u$ on  $\Omega_a^b$, with eigenvalue $\mu_{0,m}(a)$, which is constant on the boundary $\pd \Omega_a^b$. In other words, we want to find a nontrivial solution to
\begin{equation} \label{E.overdeterminedAim}
\left\{
\begin{aligned}
\ & \Delta u +\mu_{0,m}(a) u = 0 && \textup{in } \Omega_{a}^{b}\,,\\
& u = {\rm constant} \,, \quad \partial_{\nu} u  = 0 \quad&& \textup{on } \partial \Omega_{a}^{b}\,.
\end{aligned}
\right.
\end{equation}

Since finding $b \in C^{2,\alpha}(\TT)$ is a key part of the problem, we shall start by fixing a domain $\mathbb{D}$, and mapping the fixed domain $\mathbb{D}$ to the perturbed one $\Omega_a^b$. We define $\mathbb{D} := \{(\rho,\phi) \in [0,1) \times \TT\}$, and again, we identify all the points with $\rho = 0$ with the origin. We map the domain $\DD$ to $\Omega_a^b$ given in \eqref{E.Omab} via a diffeomorphism
$
\Phi_{a}^b : \DD  \to \Omab  \,,
$
defined in spherical coordinates by 
\begin{equation} \label{E.diffeomophism}
\DD \ni (\rho,\phi) \mapsto  (  \rho(\ta + \widetilde{\chi}(1-\rho) b(\phi)), \phi) \in \Omega_a^b\,.
\end{equation}
Here, $\widetilde{\chi} : \RR \to [0,1]$ is a smooth even cutoff function such that
$$
\widetilde{\chi}(t) = 1 \quad \textup{if } |t| < \frac14\,,  \quad \widetilde{\chi}(t) = 0 \quad \textup{if } |t| > \frac12 \,, \quad \textup{and} \quad |\widetilde{\chi}'(t)| \leq 10\,.
$$
Moreover, we set
\begin{equation} \label{E.pullbackAnsatz}
    \overline{\psi}_{0,m} := \psi_{0,m} \circ \Phi_{a,1}^0 \quad \textup{and} \quad \overline{\varphi}_{\ell,n}:= \varphi_{\ell,n} \circ \Phi_{a,1}^0\,.
\end{equation}

Denoting by $g_{\mathbb{S}^2}$ the standard round metric on the sphere, and setting
$$
g_b := (\Phi_a^b)^* g_{\mathbb{S}^2} \quad \textup{and} \quad v := (\Phi_a^b)^* u\,,
$$
we can rewrite \eqref{E.overdeterminedAim} as
\begin{equation} \label{E.pullbackProblem}
\left\{
\begin{aligned}
\ & L_a^b v +\mu_{0,m}(a) v = 0 && \textup{in } \DD\,,\\
& v = {\rm constant} \,, \quad \partial_{\nu_b} v  = 0 \quad&& \textup{on } \partial \DD\,,
\end{aligned}
\right.
\end{equation}
where $L_a^b$ is the Laplace-Beltrami operator  with respect to the metric $g_b$, and $\nu_{g_b}$ is the unit outer normal vector field on $\partial \DD$ with respect to $g_b$.

\begin{remark}
Recall that, in spherical coordinates $(\theta,\phi)$, the Laplace-Beltrami operator in $\mathbb{S}^2 \setminus \{ \rm N, S\}$  is given by
$$
\Delta u = \partial_\theta^2 u + \frac{\cos(\theta)}{\sin(\theta)} \pd_\theta u + \frac{1}{\sin^2(\theta)} \pd_\phi^2 u\,.
$$
Following \cite{EFR} (see also \cite[Appendix A]{EFRS}), one can check that in the coordinates $(\rho,\phi)$, and in  $\mathbb{S}^2 \setminus \{\rm N\}$,
\begin{align*}
&L_a^{b} v =  \frac{1}{(\pd_\rho \theta)^2} \pd_\rho^2 v - \frac{\pd_\rho^2 \theta}{(\pd_\rho \theta)^3} \pd_\rho v + \frac{\cos(\theta)}{\sin(\theta)} \frac{1}{\pd_\rho \theta }\pd_\rho v \\
& \quad  + \frac{1}{\sin^2(\theta)} \left[ \pd_{\phi}^2 v + \frac{(\pd_{\phi} \theta)^2}{(\pd_\rho \theta)^2} \pd_\rho^2 v - 2 \frac{\pd_\phi \theta}{\pd_\rho \theta }\pd_{\rho}\pd_{\phi} v - \frac{\pd_\rho^2 \theta (\pd_\phi \theta)^2 - 2 \pd_{\rho}\theta  \pd_{\phi} \theta (\pd_{\rho}\pd_{\phi}\theta)+(\pd_\rho \theta)^2 \pd_{\phi}^2 \theta }{(\pd_\rho \theta)^3} \pd_\rho v \right]\,,
\end{align*}
where $\theta(\rho,\phi)$ is given by \eqref{E.diffeomophism} so that
\begin{align*}
\theta (\rho, \phi) &= \rho(\ta + \widetilde{\chi}(1-\rho) b(\phi))\\
    \pd_\rho \theta(\rho,\phi) & = \ta + \widetilde{\chi}(1-\rho) b(\phi) -  \rho \widetilde{\chi}'(1-\rho) b(\phi)\,, \\
    \pd_\rho^2 \theta(\rho,\phi) & = -2\widetilde{\chi}'(1-\rho) b(\phi) + \rho \widetilde{\chi}''(1-\rho) b(\phi)\,, \\
    \pd_\phi\theta(\rho,\phi) & = \rho \widetilde{\chi}(1-\rho) b'(\phi)\,, \\
    \pd_\phi^2\theta(\rho,\phi) &= \rho \widetilde{\chi}(1-\rho) b'' (\phi)\,, \\
    \pd_\rho \pd_\phi \theta(\rho,\phi) & =  \widetilde{\chi}(1-\rho) b'(\phi) - \rho  \widetilde{\chi}'(1-\rho) b'(\phi)\,.
\end{align*}

In particular, we have that
\begin{equation}
\label{eq:La0}
L_a^0 v = \frac{1}{\ta^2} \pd_\rho^2 v + \frac{\cos(\ta \rho)}{\ta\sin(\ta \rho)}  \pd_\rho v + \frac{1}{\sin^2(\ta \rho)} \pd_\phi^2 v\,.
\end{equation}

Note that the assumption $\|b\|_{L^\infty(\T)} < \ta /10$ is enough to ensure that
$$
\theta(\rho,\phi) \in (0,\pi)\qquad \textup{for all }(\rho,\phi) \in (0,\phi) \in (0,1) \times \TT\,,
$$
and that
$$
\pd_\rho \theta (\rho,\phi) > \frac{\ta}{4} > 0\,, \quad \textup{for all }(\rho,\phi) \in [0,\phi) \in (0,1) \times \TT\,.
$$

Also, let us point out that, from $v = {\rm constant}$ on $\pd \DD$, it follows that, on $\pd \DD$,
$$
\pd_{\nu_b} v = 0 \quad \textup{if and only if} \quad \sin(\rho \ta) \pd_\rho v = 0\,.
$$
Note also that, on $\pd \DD$,
$$
\sin(\rho \ta) \pd_\rho v = 0 \ \Longleftrightarrow\ \sqrt{1-a^2} \pd_\rho v = 0\ \Longleftrightarrow\ \pd_\rho v = 0\,.
$$
\end{remark}
 
Our goal is now to find some constant $a \in (0,1)$ and some function $b \in C^{2,\al}(\TT)$ for which \eqref{E.pullbackProblem} has a nontrivial solution. Inspired by \cite{FMW}, see also \cite{EFRS}, we study \eqref{E.pullbackProblem} in anisotropic H\"older spaces. Furthermore, we are interested in looking for solutions which are $\ell$-fold symmetric, that is solutions which are invariant under the action of the isometry group of a $\ell$-sided regular polygon. Also, because of the invariance of the problem under rotations, in order to uniquely specify our solution, we impose an additional even symmetry along the $\phi = 0$ axis. Thus, we set
$$
C_\ell^{k,\alpha}(\overline{\DD}):= \big\{v \in C^{k,\alpha}(\overline{\DD}): v(\rho,\phi)=v(\rho,-\phi)\,, \  v(\rho,\phi) = v (\rho,\phi+\tfrac{2\pi}\ell)\big\}\,,
$$
which corresponds to functions with the same symmetries as $\cos (j \ell \theta)$ for $j \in \mathbb N \cup \{0\}$.

We denote the norm by 
$$
\|v\|_{C^{k,\alpha}} := \|v\|_{C^{k,\alpha}(\overline{\DD})}\,.
$$
Having this scale of Banach spaces at hand, we define
$$
\mathcal{X}^{k,\al}:= \big\{ v \in C^{k,\alpha}_\ell (\overline{\DD}): \sin(\ta \rho) \partial_{\rho} v \in C^{k,\alpha}_\ell(\overline{\DD}) \big\}\,,
$$
endowed with the norm
$
\|v\|_{\mathcal{X}^{k,\al}} := \|v\|_{C^{k,\alpha}} + \|\sin(\ta \rho)\partial_\rho v \|_{C^{k,\alpha}}
$,

and its closed subspaces 
$$
\mathcal{X}^{k,\al}\D := \big\{v \in \mathcal{X}^{k,\alpha}: v = 0 \textup{ on } \partial \DD \big\},$$
and
$$
\mathcal{X}^{k,\alpha}\DN  := \big\{v \in \mathcal{X}^{k,\alpha}: v = \pd_\rho v =0 \textup{ on } \partial \DD \big\}\,.
$$
For convenience, we also introduce the shorthand notation
$$
\mathbf{X}:= \mathcal{X}\D^{2,\alpha} \quad \textup{ and } \quad \|u\|_{\mathbf{X}} := \|u\|_{\cX^{2,\alpha}}\,.
$$
We  also need the space
$$
\mathbf{Y}:= C_\ell^{1,\alpha}(\overline{\DD}) + \mathcal{X}\D ^{0,\alpha}\,,
$$
endowed with the norm
$$
\|u\|_{\mathbf{Y}} := \inf\big\{ \|u_1\|_{C^{1,\alpha}}+\|u_2\|_{\cX^{0,\alpha}} : u_1 \in C_\ell^{1,\alpha}(\overline{\DD}),\ u_2 \in \mathcal{X}\D ^{0,\alpha},\ u = u_1+u_2 \big\}\,.
$$
As discussed in \cite[Remark 3.2]{FMW}, it is not difficult to see that $(\mathbf{Y},\|\cdot \|_{\mathbf{Y}})$ is a Banach space. Finally, we also set
$$
\mathcal{B}_{\ta} := \Big\{ b \in C^{2,\alpha}_\ell(\TT): \|b\|_{L^\infty(\TT)} < \frac{\ta}{10} \Big\}\,.
$$

We look for a solution to \eqref{E.pullbackProblem} of the form
$$
v = \overline{\psi}_{0,m} + \Theta\,,
$$
with $\overline{\psi}_{0,m}$ as in \eqref{E.eigenfunctionspsiphi} and another unknown $\Theta \in \cX^{2,\alpha}\DN$. The first result we need is the following:

\begin{lemma} \label{L.mappingProperties}
For all $a \in (0,1)$, the following assertions hold true:
\begin{itemize}
    \item[(i)] The map $(\Theta,b) \mapsto L_a^b[\,\overline{\psi}_{0,m}+\Theta] + \mu_{0,m}(a)[\,\overline{\psi}_{0,m}+\Theta]$ maps $\cX^{2,\alpha}\DN  \times \mathcal{B}_{\ta} \to\mathbf{Y}$.
    \item[(ii)] The linear operator $w \mapsto L_a^0 w + \mu_{0,m}(a) w$ maps $\mathbf{X} \to \mathbf{Y}$.
\end{itemize}
\end{lemma}

\begin{proof}
Let ($v,b) \in \cX^{2,\alpha} \times  \mathcal{B}_{\ta}$ be fixed but arbitrary. First of all, observe that
\begin{align*}
L_a^b  v & = L_a^b v - \widetilde{\chi}(\rho) L_a^0 v + \widetilde{\chi}(\rho) L_a^0 v \\
& =  \left( \frac{1}{(\pd_\rho \theta)^2} - \frac{\widetilde{\chi}(\rho)}{\ta^2} \right)\pd_\rho^2 v + \left( \frac{\cos(\theta)}{\sin(\theta)} \frac{1}{\pd_\rho \theta } - \frac{\cos(\ta \rho)}{\sin(\ta \rho)} \frac{\widetilde{\chi}(\rho)}{\ta} \right)   \pd_\rho v + \left( \frac{1}{\sin^2(\theta)} - \frac{\widetilde{\chi}(\rho)}{\sin^2(\ta \rho)} \right) \pd_\phi^2 v  + \widetilde{\chi}(\rho) L_a^0 v\\
& \quad  - \frac{\pd_\rho^2 \theta}{(\pd_\rho \theta)^3} \pd_\rho v + \frac{1}{\sin^2(\theta)} \left[\frac{(\pd_{\phi} \theta)^2}{(\pd_\rho \theta)^2} \pd_\rho^2 v - 2 \frac{\pd_\phi \theta}{\pd_\rho \theta }\pd_{\rho}\pd_{\phi} v - \frac{\pd_\rho^2 \theta (\pd_\phi \theta)^2 - 2 \pd_{\rho}\theta  \pd_{\phi} \theta (\pd_{\rho}\pd_{\phi}\theta)+(\pd_\rho \theta)^2 \pd_{\phi}^2 \theta }{(\pd_\rho \theta)^3} \pd_\rho v \right] \,,
\end{align*}
and that, by a direct computation:
\begin{equation} \label{E.commute}
\sin(\ta \rho) \pd_\rho[L_a^0 v] = L_a^0[\sin(\ta \rho) \pd_\rho v] - 2\ta \cos(\ta \rho) L_a^0v\,.
\end{equation}

Having this decomposition at hand, we define
\begin{align*}
    F_1(v,b)  :=\ &  \left( \frac{1}{(\pd_\rho \theta)^2} - \frac{\widetilde{\chi}(\rho)}{\ta^2} \right) \pd_\rho^2 v + \left( \frac{\cos(\theta)}{\sin(\theta)} \frac{1}{\pd_\rho \theta } - \frac{\cos(\ta \rho)}{\sin(\ta \rho)} \frac{\widetilde{\chi}(\rho)}{\ta} \right)   \pd_\rho v \\
 & +\frac{1}{\sin^2(\theta)} \left(   \frac{(\pd_{\phi} \theta)^2}{(\pd_\rho \theta)^2} \pd_\rho^2 v - 2 \frac{\pd_\phi \theta}{\pd_\rho \theta} \pd_\rho \pd_\phi v \right)  - \frac{\pd_\rho^2 \theta}{(\pd_\rho \theta)^3} \pd_\rho v  \\  
    F_2(v,b) :=\ & \left( \frac{1}{\sin^2(\theta)} - \frac{\widetilde{\chi}(\rho)}{\sin^2(\ta \rho)} \right) \pd_\phi^2 v + \widetilde{\chi}(\rho) L_a^0 v \\
    & - \frac{1}{\sin^2(\theta)} \left( \frac{\pd_\rho^2 \theta (\pd_\phi \theta)^2 - 2 \pd_{\rho}\theta  \pd_{\phi} \theta (\pd_{\rho}\pd_{\phi}\theta)+(\pd_\rho \theta)^2 \pd_{\phi}^2 \theta }{(\pd_\rho \theta)^3} \pd_\rho v \right)\,,
\end{align*}
and we point out that
$$
L_a^b v = F_1(v,b) + F_2(v,b)\,,
$$
and that all the terms of $F_1$ and $F_2$ except for $\tilde \chi (\rho ) L_a^0 v$ vanish for $\rho < 1/4$ (since in that range $\tilde \chi = 1$, and $\theta = \tilde a \rho$).

Moreover, it is easy to check that $F_1(v,b) + \mu_{0,m}(a) v \in C_\ell^{1,\alpha}(\overline{\DD})$ and $F_2(v,b) \in \cX^{0,\al}\D$ for $v = \overline{\psi}_{0,m} + \Theta$ with $\Theta \in \cX^{2,\alpha}\DN$.
This concludes the proof of $(i)$. 

Let $w \in \mathbf{X}$ be fixed but arbitrary. To prove $(ii)$, we first observe that
$$
L_a^0w = \frac{\sin^2(\ta \rho)}{1-a^2} L_a^0 w + \left( 1-  \frac{\sin^2(\ta \rho)}{1-a^2} \right) L_a^0w\,,
$$
where the division by $1-a^2$ is motivated by noticing that at $\rho = 1$ we have $\sin^2 (\tilde a ) = \sin^2 (\arccos (a)) = 1- a^2$. Then, we set
\begin{align*}
    F_3(w) & :=  \frac{\sin^2(\ta \rho)}{1-a^2} L_a^0 w - \frac{1}{1-a^2} \pd_\phi^2 w\,, \\
    F_4(w) & :=  \frac{1}{1-a^2} \pd_\phi^2 w + \left( 1-  \frac{\sin^2(\ta \rho)}{1-a^2} \right) L_a^0w\,,
\end{align*}
and stress that
$$
L_a^0 w  = F_3(w) + F_4(w)\,.
$$
As before, one can check that $F_3(w) + \mu_{0,m}(a) w \in C^{1,\alpha}_\ell(\overline{\DD})$ and $F_4(w) \in \cX^{0,\alpha}\D$ for any $w \in \mathbf{X}$, concluding the proof of $(ii)$. 
\end{proof}

We now construct the bifurcation operator. Inspired by \cite{FMW, EFRS}, we reduce our problem so that the basic unknowns are just a constant $a \in (0,1)$ and a function $v \in \mathbf{X}$. More precisely, we design an operator $G_a: \mathbf{O} \to \mathbf{Y}$, where $\mathbf{O}$ is an open neighborhood of the origin in $\mathbf{X}$, such that
$$
G_a(v) = 0\,, \quad v \in \mathbf{X}\,,
$$
encapsulates the entire overdetermined problem \eqref{E.pullbackProblem}. To that end, for each $v \in \mathbf{X}$, we define
\begin{equation} \label{E.corrections}
    b_v(\phi) := - \frac{\ta}{ \overline{\psi}_{0,m}''(1)} \, \pd_\rho v (1,\phi)\quad \textup{and}  \quad w_v(\rho,\phi) := v(\rho,\phi) + \frac{1}{\ta} \,\overline{\psi}_{0,m}'(\rho) \rho \widetilde{\chi}(\rho-1) b_v(\phi)\,,
\end{equation}
and introduce the open subset
$$
\mathbf{O}:= \Big\{v \in \mathbf{X} :   \|b_v\|_{L^\infty(\TT)} < \frac{\ta}{10} \Big\}\,.
$$

The bifurcation operator we are going to consider is
$$
G_a: \mathbf{O} \to \mathbf{Y}, \quad v \mapsto L_a^{b_v}[\,\overline{\psi}_{0,m}+w_v] + \mu_{0,m}(a) [\,\overline{\psi}_{0,m}+w_v]\,.
$$
Note that the inclusion $G_a(\mathbf{O}) \subset \mathbf{Y}$ follows from Lemma \ref{L.mappingProperties} and the fact that $v \in \mathbf{X}$ implies $w_v \in \cX^{2,\alpha}\DN$. At this point, it should be noted that if $v \in \mathbf{X}$ satisfies $G_a(v) = 0$ for some $a \in (0,1)$, then $\widetilde{w} := \overline{\psi}_{0,m} + w_v$ is the desired solution to \eqref{E.pullbackProblem} with $b = b_v$. The proof of Theorem \ref{T.bifurcation} is now reduced to finding a uniparametric family of nontrivial zeros of the operator $G_a$. This will be done in the following subsections using the celebrated Crandall--Rabinowitz theorem, see, for instance, \cite{CR} or \cite[Theorem I.5.1]{Kielhofer}. Let us also stress that in this setting
$$
G_a(0) = 0\,, \quad \textup{for all } a \in (0,1)\,.
$$
This is the branch of trivial solutions from which we will bifurcate.

\subsection{The linearized operator}

The Fr\'echet derivative of $G_a$ plays a major role in local bifurcation arguments. In our next step towards the proof of Theorem \ref{T.bifurcation}, we compute the Fr\'echet derivative of $G_a$ and analyze its properties.  

\begin{lemma}
For all $a \in (0,1)$, the map $G_a: \mathbf{O} \to \mathbf{Y}$ is smooth. Moreover,
$$
DG_a(0)v = L_a^0 v + \mu_{0,m}(a) v\quad  \textup{ for all } v \in \mathbf{X}\,,
$$
where $L_a^0$ is the Laplace-Beltrami operator with respect to the metric $g_0$. 
\end{lemma}

\begin{proof}
    The smoothness of $G_a$ immediately follows from its definition, so we just have to prove that $DG_a(0) = L_a^0 + \mu_{0,m}(a)\, {\rm id}$. Since the maps $v \mapsto b_v$ and $v \mapsto w_v$ are linear, we have that $b_{sv} = s b_v$ and $w_{sv} = s w_v$, and so that
\begin{equation} \label{E.ConcDGa1}
        \begin{aligned}
    DG_a(0) v & = \frac{\dd}{\dd s} \left( L_a^{s b_w} [\,\overline{\psi}_{0,m} + s w_v] + \mu_{0,m}(a) [\, \overline{\psi}_{0,m} + s w_v] \right)\Big|_{s=0} \\
& = L_a^0 w_v + \mu_{0,m}(a)w_v + \frac{\dd}{\dd s} L_a^{sb_v}\Big|_{s=0}[\,\overline{\psi}_{0,m}]\,.
    \end{aligned}
\end{equation}

Now, to compute the last term on the second line, let us point out that $\psi_{0,m}$ actually satisfies
$$
\Delta \psi_{0,m} + \mu_{0,m}(a) \psi_{0,m} = 0 \quad \textup{on } \mathbb{S}_+^2\,.
$$
Thus, setting $\psi_{0,m}(\theta,\phi) \equiv \psi_{0,m}(\theta)$ by abuse of notation, and defining $\overline{\psi}_{0,m}^{b} := (\Phi_a^b)^*\psi_{0,m} $, we have that
$$
L_a^b \overline{\psi}_{0,m}^b + \mu_{0,m}(a) \overline{\psi}^b_{0,m} = 0\,, \quad \textup{in } \DD\,.
$$
In particular, it follows that
$$
L_a^{sb_v} \overline{\psi}_{0,m}^{sb_v} + \mu_{0,m}(a) \overline{\psi}^{sb_v}_{0,m} = 0\,, \quad \textup{in } \DD\,,
$$
and so that
\begin{align*}
    0 & = \frac{\dd}{\dd s} \left( L_a^{sb_v} \overline{\psi}_{0,m}^{sb_v} + \mu_{0,m}(a) \overline{\psi}^{sb_v}_{0,m} \right)\Big|_{s=0} =  \frac{\dd}{\dd s} L_a^{sb_v}\Big|_{s=0}[\,\overline{\psi}_{0,m}] + L_a^0\left[ \frac{\dd}{\dd s} \overline{\psi}_{0,m}^{sb_v} \Big|_{s=0} \right] + \mu_{0,m}(a) \frac{\dd}{\dd s} \overline{\psi}_{0,m}^{sb_v}\Big|_{s=0}
    \,.
\end{align*}
A direct computation shows that
$$
\frac{\dd}{\dd s} \overline{\psi}_{0,m}^{sb_v} \Big|_{s=0} = w_v -v \,.
$$
Hence, it follows that
\begin{equation} \label{E.ConcDGa2}
   \frac{\dd}{\dd s} L_a^{sb_v}\Big|_{s=0}[\,\overline{\psi}_{0,m}] + L_a^0 [w_v-v] + \mu_{0,m}(a) [w_v-v] = 0\,.
\end{equation}
The result follows from \eqref{E.ConcDGa1} and \eqref{E.ConcDGa2} using the fact that $L_a^0$ is a linear operator. 
\end{proof}

Having this expression for the linearized operator at hand, we prove that it is a Fredholm operator of index zero. As a preliminary step, we prove the following regularity result:

\begin{lemma} \label{L.regularity}
    For all $a \in (0,1)$, let $f \in C_\ell ^{0,\alpha}(\overline{\DD})$ and $v \in C_\ell ^{2,\alpha}(\overline{\DD})$ satisfy
\begin{equation} \label{E.regularity}
L_a^0 v = f \quad \textup{in }\DD\,,\qquad v=0\quad \textup{on }\pd\DD\,.
\end{equation}
If $f \in \mathbf{Y}$, then $v \in \mathbf{X}$. 
\end{lemma}

\begin{proof}
    First, let us point out that, by the definition of $\mathbf{X}$, the proof reduces to showing that $ \vartheta :=\sin(\ta \rho) \pd_\rho v \in C^{2,\alpha}_\ell(\overline{\DD})$. Taking into account \eqref{E.commute}, we see that $\vartheta$ satisfies
    \begin{equation} \label{E.vartheta}
    L_a^0 \vartheta = \sin(\ta\rho) \pd_\rho f + 2\ta \cos(\ta\rho) L_a^0 v \quad \textup{in } \DD\,,
    \end{equation}
    in the distributional sense. Moreover, using the definition of $\mathbf{Y}$, it is immediate to check that the right hand side belongs to $C_\ell^{0,\alpha}(\overline{\DD})$. 

    Now, taking into account the expression of $L_a^0$ in spherical coordinates \eqref{eq:La0}, the fact that $v = 0$ on $\pd \DD$, and \eqref{E.regularity}, we can check that
    $$
    \pd_\rho \vartheta = \sin (\tilde a ) \partial_\rho^2 v + \tilde a \cos (\tilde a ) \partial_\rho v = \ta^2\sin (\ta ) L_a^0 v  =  \ta^2 \sqrt{1-a^2}\, f \quad \textup{on } \pd\DD\,.
    $$
    Moreover, writing $f = f_1 + f_2$ with $f_1 \in C^{1,\alpha}_\ell(\overline{\DD})$ and $f_2 \in \cX^{0,\alpha}\D$\,, we get that
    \begin{equation} \label{E.boundaryvartheta}
        \pd_\rho \vartheta = \ta^2 \sqrt{1-a^2} f_1 \quad \textup{on } \partial\DD\,.
    \end{equation}
    Standard regularity for the Neumann problem \eqref{E.vartheta}--\eqref{E.boundaryvartheta} allows us to conclude that $\vartheta \in C^{2,\alpha}(\overline{\DD})$, as desired. 
\end{proof}

\begin{lemma} \label{L.Fredholm}
    For all $a \in (0,1)$, $DG_a(0) : \mathbf{X} \to \mathbf{Y}$ is a Fredholm operator of index zero. 
\end{lemma}

\begin{proof}
    First of all, observe that the multiplication operator $M : \mathbf{X} \to \mathbf{Y},\ v \mapsto \mu_{0,m}(a) v$, is compact. This is an immediate consequence of the compactness of the embedding $\mathbf{X} \hookrightarrow \mathbf{Y}$, which in turn follows from standard compact embedding results for H\"older spaces. Hence, since the property of being a Fredholm operator and the Fredholm index of an operator are preserved by compact perturbations, it suffices to show that $L_a^0 : \mathbf{X} \to \mathbf{Y}$ defines a Fredholm operator of index zero. We will prove the stronger result that $L_a^0$ is a topological isomorphism. To that end, since $L_a^0$ is a continuous linear map, by the open mapping theorem, it suffices to show that it is a bijective map.

    We start showing injectivity. We set $w:= ((\Phi_a^0)^{-1})^* v  $ and $g:= ((\Phi_a^0)^{-1})^* f$. By the definition of the pullback, \eqref{E.regularity} is equivalent to
    \begin{equation} \label{E.PoissonPhysical}
    \Delta w = g \quad \textup{in } \Omega_a\,,  \qquad w = 0 \quad\textup{on } \pd \Omega_a\,.
    \end{equation}
    For any $g \in C^{0,\alpha}(\overline{\Omega}_a)$, the existence and uniqueness of solutions to \eqref{E.PoissonPhysical} belonging to $C^{2,\alpha}(\overline{\Omega}_a)$ is classical. In particular, if $g \equiv 0$, then the unique solution to \eqref{E.PoissonPhysical} is $w \equiv 0$. This shows that $L_a^0: \mathbf{X} \to \mathbf{Y}$ is an injective map. 
    
    Now, we prove that $L_a^0: \mathbf{X} \to \mathbf{Y}$ is onto. We just have to realize that the equivalence above implies that, for any $f \in C_\ell^{0,\alpha}(\overline{\DD})$, there exists a unique solution to \eqref{E.regularity} belonging to $C_\ell^{2,\alpha}(\overline{\DD})$. If in addition $f \in \mathbf{Y}$, Lemma \ref{L.regularity} implies that $v \in \mathbf{X}$, and so that $L_a^0: \mathbf{X} \to \mathbf{Y}$ is onto.  
\end{proof}

\subsection{Proof of Theorem \ref{T.bifurcation}}

We can now provide the proof of Theorem \ref{T.bifurcation}. First of all, under the assumptions of Theorem \ref{T.bifurcation} we check that the kernel of the linearized operator is one-dimensional, and that the transversality condition in the Crandall-Rabinowitz theorem holds. More precisely, we have the following:

\begin{proposition} \label{P.CR}
        Assume that there exists $a_\star \in (0,1)$ such that:
    \begin{itemize}
        \item[$(i)$] $\mu_{0,m}(a_\star) = \lambda_{\ell,n}(a_\star)\quad$ (Bifurcation)
        \item[$(ii)$] $\mu'_{0,m}(a_\star) \neq \lambda'_{\ell,n}(a_\star)\quad $ (Transversality)
        \item[$(iii)$] $\lambda_{\ell,n}(a_\star) \neq \lambda_{\overline{m}\ell,\overline{n}}(a_\star)$ \textup{ for all nonnegative integers} $\overline{m}, \overline{n}$ with $(\overline{m},\overline{n}) \neq (1,n) \quad$ (Nonresonance)  
    \end{itemize}
Moreover, let $\overline{\varphi}_{\ell,n}^{\,a} = \varphi_{\ell,n} \circ \Phi_{a,1}^0$ be as in \eqref{E.pullbackAnsatz}. Then:
\begin{itemize}
    \item[(a)] The kernel of $DG_{a_\star}(0)$ is one-dimensional. Moreover,
    $$
    \textup{Ker}(DG_{a_\star}(0)) = \textup{span}\{\overline{v}_{a_\star}\}\,, \quad \textup{where} \quad \overline{v}_{a_\star}(\rho,\phi) = \overline{\varphi}_{\ell,n}^{\,a_\star}(\rho) \cos(\ell\phi)\,.
    $$
    \item[(b)] $DG_{a_\star}(0)$ satisfies the transversality property, that is
    $$
    \frac{\dd }{\dd a} DG_{a}(0) \Big|_{a=a_\star}[\,\overline{v}_{a_\star}] \not\in \textup{Im}(DG_{a_\star}(0))
    $$
\end{itemize}
\end{proposition}

\begin{proof}
    $(a)$ Let $v \in \textup{Ker} (DG_{a_\star}(0))$ and let $w :=  ((\Phi_{a_\star}^0)^{-1})^*v$. Then,
    $$
    \Delta w + \mu_{0,m}(a_\star) w = 0 \quad \textup{in } \Omega_{a_\star}\,, \quad w = 0 \quad \textup{on } \pd \Omega_{a_\star}\,.
    $$ 
    By the bifurcation condition $(i)$ we know that $\mu_{0,m}(a_\star) = \lambda_{\ell,n}(a_\star)$. On the other hand, by the nonresonance condition $(iii)$, we know that $\lambda_{\ell,n}(a_\star) \neq \lambda_{\overline{m}\ell,\overline{n}}(a_\star)$ for all nonnegative integers $\overline{m}, \overline{n}$ with $(\overline{m},\overline{n}) \neq (1,n)$. Combining the definition of $\mathbf{X}$, which only allows for $\ell-$fold symmetric functions, with these two conditions, we conclude that $w(\theta,\phi) = t \varphi_{\ell,n}(\theta)\cos(\ell\phi)$ with $t \in \RR$. Hence, $(a)$ follows.

    $(b)$ We first show that
\begin{equation} \label{E.ImageLinearized}
    \textup{Im}(DG_{a_\star}(0)) = \Big\{ w \in \mathbf{Y}: \ \langle\overline{v}_{a_\star}, w \rangle_{L^2(\Omega_{a_\star},\,\sin(a_\star\rho) d\rho d\phi)} = 0 \Big\}\,.
\end{equation}
    For simplicity, through the rest of the proof we use the notation
    $$
    \langle \cdot , \cdot \rangle :=  \langle \cdot , \cdot \rangle_{L^2(\Omega_{a_\star},\,\sin(a_\star\rho) d\rho d\phi)}\,.
    $$
By Lemma \ref{L.Fredholm} we know that $DG_{a_\star}(0)$ is a Fredholm operator of index zero. Hence, we only need to prove that
$$
    \textup{Im}(DG_{a_\star}(0)) \subset \Big\{ w \in \mathbf{Y}: \ \langle\overline{v}_{a_\star}, w \rangle = 0 \Big\}\,.
$$
    Let $w \in \textup{Im}(DG_{a_\star}(0))$ be fixed but arbitrary. Then, there exists $\vartheta \in \mathbf{X}$ such that $DG_{a_\star}(0) \vartheta = w$. Moreover, since $\overline{v}_{a_\star} \in \textup{Ker}(DG_{a_\star}(0))$, we can integrate by parts and obtain that
    $$
    \langle w, \overline{v}_{a_\star} \rangle = \langle DG_{a_\star}(0) \vartheta , \overline{v}_{a_\star} \rangle = \langle \vartheta , DG_{a_\star}(0) \overline{v}_{a_\star} \rangle = 0\,.
    $$
    This implies that \eqref{E.ImageLinearized} holds, as desired.

    Now, to conclude the proof of $(b)$, we just have to show that
    $$
    \left\langle  \frac{\dd }{\dd a} DG_{a}(0) \Big|_{a=a_\star}[\,\overline{v}_{a_\star}], \overline{v}_{a_\star} \right\rangle \neq 0\,.
    $$
    Let $\overline{v}_a(\rho,\phi) := \overline{\varphi}_{\ell,n}^a(\rho) \cos(\ell\phi)$ and $\overline{w}_{a_\star} := \frac{\dd}{\dd a} \overline{v}_a\big|_{a=a_\star}$. Since
    $$
    DG_a(0) \overline{v}_a = L_a^0 \overline{v}_a + \mu_{0,m}(a) \overline{v}_a = (\mu_{0,m}(a) - \lambda_{\ell,n}(a)) \overline{v}_a\,,
    $$
    we can use again the bifurcation condition $(i)$ and obtain that
    $$
    \frac{\dd}{\dd a} DG_a(0)\Big|_{a = a_\star} [\, \overline{v}_{a_\star}] + DG_{a_\star}(0) \overline{w}_{a_\star} =  (\mu_{0,m}'(a_\star) - \lambda_{\ell,n}'(a_\star)) \overline{v}_{a_\star}\,.
    $$
    Moreover, observe that
    $$
    \langle DG_{a_\star}(0) \overline{w}_{a_\star},  \overline{v}_{a_\star} \rangle = \langle  \overline{w}_{a_\star},  DG_{a_\star}(0) \overline{v}_{a_\star} \rangle = 0\,.
    $$
    Hence, we have that
    $$
    \left\langle   \frac{\dd}{\dd a} DG_a(0)\Big|_{a = a_\star} [\, \overline{v}_{a_\star}], \overline{v}_{a_\star} \right\rangle = (\mu_{0,m}'(a_\star) - \lambda_{\ell,n}'(a_\star)) \langle \overline{v}_{a_\star}, \overline{v}_{a_\star} \rangle\,.
    $$
    The desired conclusion immediately follows from the transversality condition $(ii)$.  
\end{proof}

\begin{proof}[\textbf{Proof of Theorem \ref{T.bifurcation} (completed)}]
    By the definition of $G_a(0)$, Lemma \ref{L.Fredholm} and Proposition \ref{P.CR} we know that the map
\begin{equation}
\big(0,1\big) \times \mathcal{O} \to \cY\,, \qquad (a,v) \mapsto G_a(v)\,,
\end{equation}
satisfies the hypotheses of the Crandall--Rabinowitz theorem. Therefore, there exits a nontrivial continuously differentiable curve through $(a_\star,0)$,
$$
\big\{(a_s, v_s): s \in (-s_\star ,s_\star ),\ (a_0, v_0) = (a_\star,0) \big\}\subset (0,1)\times \mathcal O\,,
$$
such that
$$
G_{a_s}(v_s) = 0\,, \quad \textup{for} \quad s \in  (-s_\star ,s_\star )\,.
$$
Moreover, for $\overline{v}_{a_\star}$ as in Proposition \ref{P.CR} $(a)$, it follows that
\begin{equation} \label{E.expansionvs}
v_s = s \, \overline{v}_{a_\star} + o(s) \quad \textup{in } \mathbf{X}\,, \quad \textup{as } s \to 0\,.
\end{equation}

In this case, we find that
\begin{equation} \label{E.wstilde}
\widetilde{w}_s = \overline{\psi}_{0,m}^{\,a_s} + w_{v_s}\,,  \quad \textup{for} \quad s \in  (-s_\star ,s_\star)\,,
\end{equation}
is a non-trivial solution to \eqref{E.pullbackProblem} with $b = b_{v_s}$. Taking into account the definitions of $b_{v_s}$ and $w_{v_s}$, see \eqref{E.corrections}, we get the desired expression for $b_s(\phi) := b_{v_s}(\theta)$ in \eqref{E.bs}. To return to the original variables, we simply choose $u_s := ((\Phi_{a_s}^{b_s})^{-1})^*\widetilde{w}_s$ and conclude that, for all $s \in (-s_\star,s_\star)$, $u_s \in C^{2,\alpha}(\overline{\Omega}_{a_s}^{b_s})$ solves \eqref{E.overdeterminedCR}.

Finally, it follows from the regularity results of Kinderlehrer and Nirenberg \cite{KN} that, in fact, both the domains $\Omega_{a_s}^{b_s}$ and the solutions $u_s$ are real analytic, for all $s \in (-s_\star,s_\star)$.
\end{proof}

\subsection{Proof of the main results}

At this point, it is immediate to prove the main results of the paper, namely Theorem \ref{T.intro} and Corollary \ref{C.intro}.

\begin{proof}[\textbf{Proof of Theorem \ref{T.intro}}]
    The result immediately follows from the combination of Proposition \ref{P.ComputerAssisted} and Theorem \ref{T.bifurcation}. Indeed, let $a_\star \in (0,1)$ and $m_\star \in \mathbb{N}$, $m_\star \geq 2$, be given by Proposition \ref{P.ComputerAssisted}. The result follows from Theorem \ref{T.bifurcation} applied with $\ell = 8$, $m = m_\star$ and $n = 0$.
\end{proof}

\begin{proof}[\textbf{Proof of Corollary \ref{C.intro}}]
    Let $(u_s)_s$ be the family solutions to \eqref{E.OverdeterminedM} given by Theorem \ref{T.intro} (see also Theorem \ref{T.bifurcation}), and $(\Omega_s)_s \subset \mathbb R^2$ the corresponding family of domains. We set $$\omega_s := \frac{u_s }{\psi_{0,m}(\ta_s)} - 1\,, \quad \textup{for all } 0 < |s| \ll 1\,,$$ 
    and $(\omega_s)_s$ is precisely the family of nontrivial sign-changing solutions to \eqref{E.OverdeterminedCorollaryIntro} that we were looking for. Indeed, using that $u_s|_{\partial \Omega} = \psi_{0,m}(\ta_s) =  constant$, and that it is a Neumann eigenfunction on $\Omega_s$, it is immediate to check that $\omega_s$ is a solution to \eqref{E.OverdeterminedCorollaryIntro}. The fact that $\omega_s$ changes sign immediately follows from the behavior of $\omega_s$ at main order (see \eqref{E.expansionvs} and \eqref{E.wstilde}) and the fact that $m_\star \geq 2$. 
    We refer to the proof of Theorem \ref{T.bifurcation} for more details.  
 
\end{proof}

\begin{figure}[htbp] 
    \centering 

    \begin{subfigure}[b]{0.48\textwidth}
        \centering
        \includegraphics[width=\linewidth]{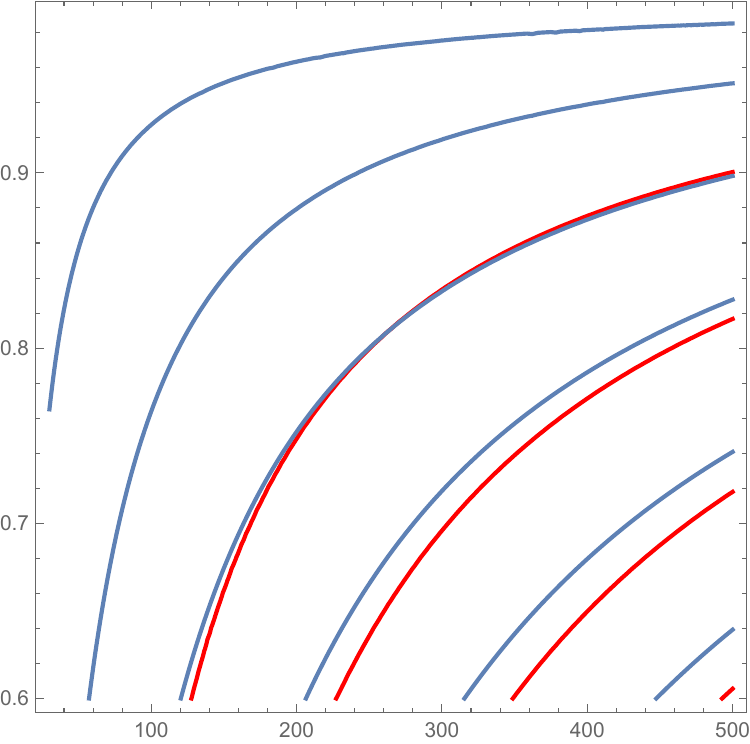}
        \caption{$\mathbb{S}^2$ with $\ell_{\text{Dir}} = 6$.}
        \label{fig:a}
    \end{subfigure}
    \hfill 
    \begin{subfigure}[b]{0.48\textwidth}
        \centering
        \includegraphics[width=\linewidth]{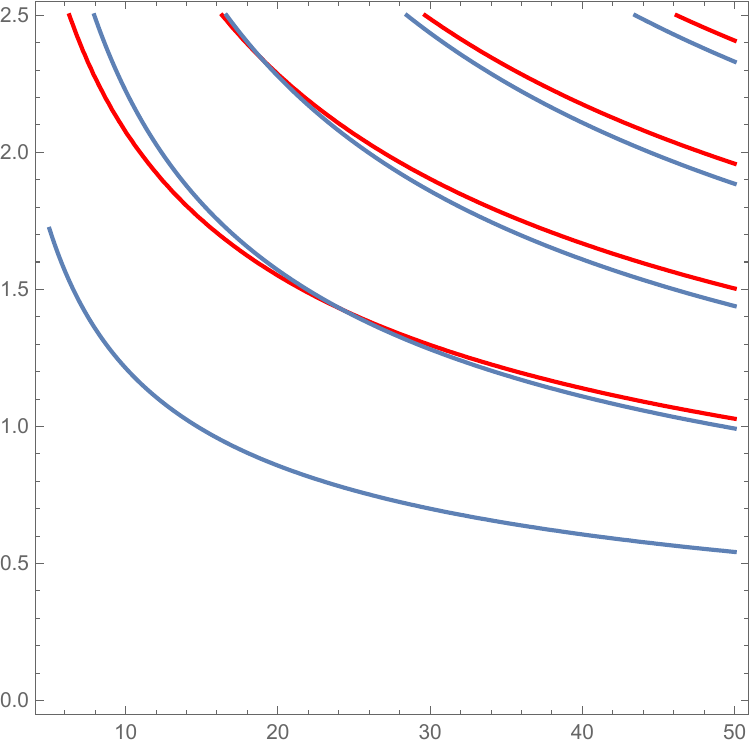}
        \caption{$\mathbb{H}^2$ with $\ell_{\text{Dir}} = 4$.}
        \label{fig:b}
    \end{subfigure}

    \vspace{0.5cm} 

    \begin{subfigure}[b]{0.48\textwidth}
        \centering
        \includegraphics[width=\linewidth]{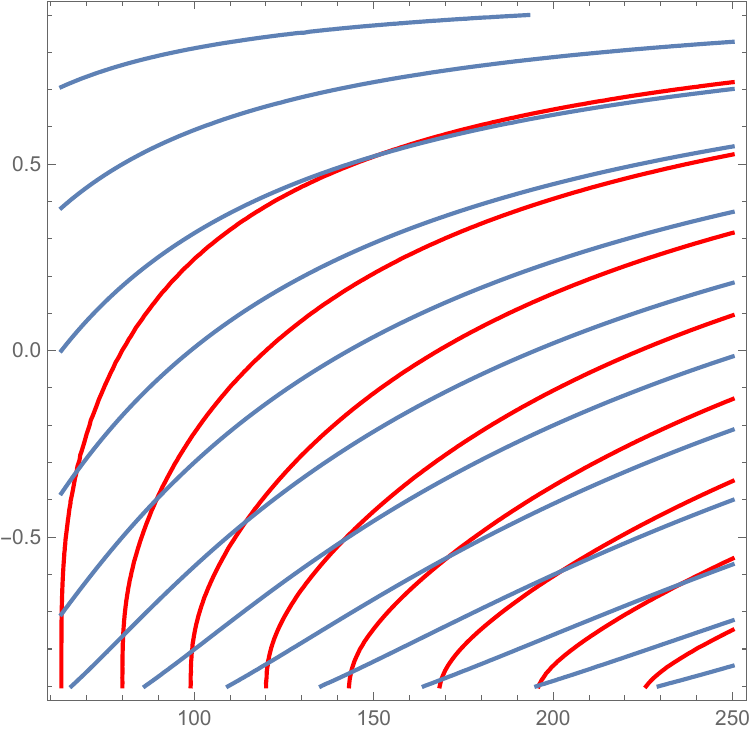}
        \caption{$\mathbb S^3$ with $\ell_{\text{Dir}} = 7$}
        \label{fig:c}
    \end{subfigure}
    \hfill 
    \begin{subfigure}[b]{0.48\textwidth}
        \centering
        \includegraphics[width=\linewidth]{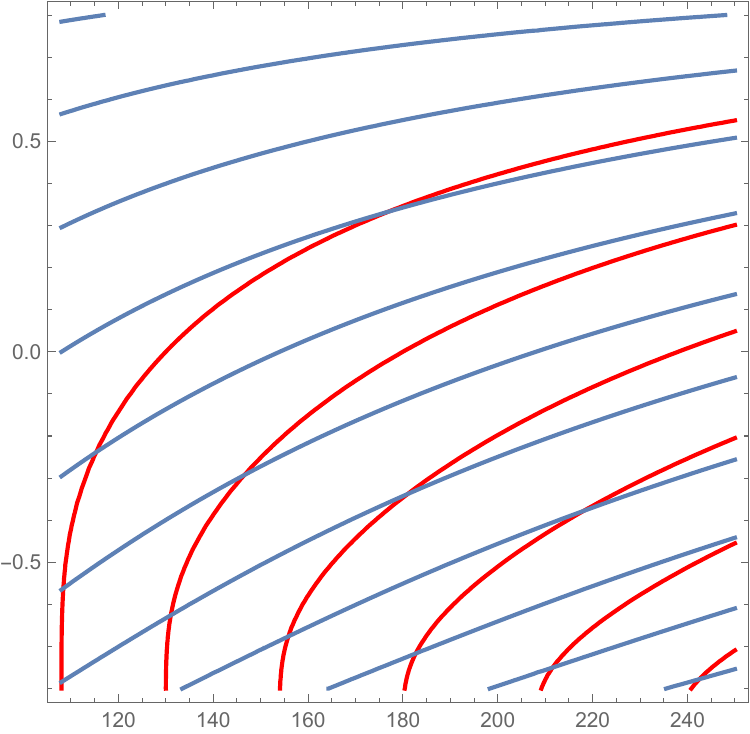}
        \caption{$\mathbb S^4$ with $\ell_{\text{Dir}} = 9$.}
        \label{fig:d}
    \end{subfigure}

    \caption{Crossings between zonal Neumann eigenvalues (red) and non-zonal Dirichlet ones (blue) with angular mode $\cos (\ell_{\text{Dir}} \phi )$, for different domains. The horizontal axis denotes the eigenvalue $\lambda$. The vertical one is the height a of geodesic ball centered around the north pole for all the unit sphere cases. In the case of Figure \ref{fig:b}, we use polar coordinates $(r, \varphi)$ on the hyperbolic plane so that $\Delta_{\mathbb H^2} = \frac{\partial^2}{\partial r^2} + \coth(r) \frac{\partial}{\partial r} - \frac{1}{\sinh (r)^2} \frac{\partial^2}{\partial \phi^2}$ and the vertical axis of Figure \ref{fig:b} denotes the radius $r$ of our geodesic ball centered at $0$.}
    \label{fig:appendix}
\end{figure}

\section*{Acknowledgments} 
This work was initiated during the workshops \textit{PDE in Barcelona}
at the Universtat de Barcelona, and \textit{Oberwolfach Workshop Partial Differential Equations} at the MFO. The authors would like to thank the organizers of the workshops for the kind invitations. Likewise, they would like to thank T. Weth for the enlightening discussions at the MFO. 

G.C.L. was supported by the Swiss State Secretariat for Education, Research and Innovation (SERI) under contract number MB22.00034 through the project TENSE. G.C.L. was also partially supported by the MICINN research grant number PID2021–125021NA–I00. A.J.F. is partially supported by the grants PID2023-149451NA-I00 of MCIN/AEI/10.13039/ 501100011033/FEDER, UE and Proyecto de Consolidaci\'on Investigadora 2022, CNS2022-135640, MICINN (Spain). 

\appendix
\section{The basics about isoparametric functions and hypersurfaces} \label{A.geometry} 

In this short appendix, for the benefit of the reader, we recall the notions of homogeneous and isoparametric hypersurfaces, and review some of their classical properties. We refer to \cite{S} and the references therein for a more general background. If one prefers so, they could consult the lecture notes \cite{DV}.

\begin{definition} \label{D.homogeneous}
 Let $(M,g)$ be a Riemannian manifold and let $S$ be a regular $(C^k,$ smooth or analytic) and connected hypersurface in $M$. $S$ is called a \textit{homogeneous hypersurface} in $M$ if it is an orbit of a subgroup $G \subset \textup{ISO}(M)$, where $\textup{ISO}(M)$ is the group of all isometries of $M$.
\end{definition}

\begin{remark}
    The principal curvatures of a homogeneous hypersurface are constant. However, the converse statement is not always true. See \cite[Page 532]{S} for a counterexample. 
\end{remark}
 
\begin{definition} \label{D.isoparametric}
 Let $(M,g)$ be a Riemannian manifold, let $U \subseteq M$ be an open connected subset, and let $f: U \to \R$ be a smooth function. We say that $f$ is an \textit{isoparametric function} on  $U$ if there exists smooth functions $h_1$ and $h_2$, defined on $f(U) \subset \R$, so that
 $$
 \Delta_g f = h_1 \circ f \quad \textup{and} \quad |\nabla f|^2 = h_2 \circ f\,.
  $$
 Regular level sets of $f$ are called \textit{isoparametric hypersurfaces} in $M$. 
\end{definition}

\begin{proposition}
     Let $(M,g)$ be a Riemannian manifold. If a hypersurface $S$ is homogeneous in $M$, then it is isoparametric in $M$.
\end{proposition}
 
\begin{theorem}[Cartan's Theorem] \label{T.Cartan}
    Let $(M,g)$ be a complete Riemannian manifold of constant sectional curvature. A hypersurface $S \subset M$ is isoparametric in $M$ if and only if all its principal curvatures are constant. 
\end{theorem}

\bibliographystyle{siam}

\vspace{0.7cm}

\noindent \textbf{Gonzalo Cao-Labora}  \smallbreak

Institute of Mathematics, EPFL, Station 8, 1015 Lausanne VD, Switzerland

\textit{Email address \rm{:}} \texttt{gonzalo.caolabora@epfl.ch}

\bigbreak

\noindent \textbf{Antonio J. Fern\'andez}  \smallbreak

Departamento de Matem\'aticas, Universidad Aut\'onoma de Madrid, 28049 Madrid, Spain

\textit{Email address \rm{:}} \texttt{antonioj.fernandez@uam.es}

\end{document}